\documentclass[11pt, reqno]{amsart}
\usepackage{appendix}
\usepackage[a4paper, centering]{geometry}
\usepackage{xcolor}
\usepackage[utf8]{inputenc}
\usepackage[colorlinks=true,hyperindex,pagebackref]{hyperref}
\usepackage{cleveref}
\usepackage{amsfonts}
\usepackage{amssymb}
\usepackage{tikz-cd}
\usepackage[T1]{fontenc}
\hypersetup{
	colorlinks,
	linkcolor=violet,
	citecolor=cyan,
	urlcolor=orange
}
\usepackage{stmaryrd}
\usepackage{mathrsfs}  
\usepackage{varwidth}
\theoremstyle{plain}
\newtheorem{theorem}{Theorem}[section]
\newtheorem{proposition}[theorem]{Proposition}
\newtheorem{lemma}[theorem]{Lemma}
\newtheorem{corollary}[theorem]{Corollary}

\theoremstyle{definition}
\newtheorem{definition}[theorem]{Definition}
\newtheorem{example}[theorem]{Example}

\newcommand{\nc}{\newcommand}
\nc{\on}{\operatorname}

\nc{\Q}{\mathbb{Q}}
\nc{\Z}{\mathbb{Z}}
\nc{\cl}{\mathrm{cl}}

\nc{\fraka}{{\mathfrak a}} \nc{\bba}{{\mathbf a}}
\nc{\frakb}{{\mathfrak b}}
\nc{\frakc}{{\mathfrak c}}
\nc{\frakd}{{\mathfrak d}}
\nc{\frake}{{\mathfrak e}}
\nc{\frakf}{{\mathfrak f}}
\nc{\frakg}{{\mathfrak g}}
\nc{\frakh}{{\mathfrak h}}
\nc{\fraki}{{\mathfrak i}}
\nc{\frakj}{{\mathfrak j}}
\nc{\frakk}{{\mathfrak k}}
\nc{\frakl}{{\mathfrak l}}
\nc{\frakm}{{\mathfrak m}}
\nc{\frakn}{{\mathfrak n}}
\nc{\frako}{{\mathfrak o}}
\nc{\frakp}{{\mathfrak p}}
\nc{\frakq}{{\mathfrak q}}
\nc{\frakr}{{\mathfrak r}}
\nc{\fraks}{{\mathfrak s}}
\nc{\frakt}{{\mathfrak t}}
\nc{\fraku}{{\mathfrak u}}
\nc{\frakv}{{\mathfrak v}}
\nc{\frakw}{{\mathfrak w}}
\nc{\frakx}{{\mathfrak x}}
\nc{\fraky}{{\mathfrak y}}
\nc{\frakz}{{\mathfrak z}}
\nc{\frakA}{{\mathfrak A}}
\nc{\frakB}{{\mathfrak B}}
\nc{\frakC}{{\mathfrak C}}
\nc{\frakD}{{\mathfrak D}}
\nc{\frakE}{{\mathfrak E}}
\nc{\frakF}{{\mathfrak F}}
\nc{\frakG}{{\mathfrak G}}
\nc{\frakH}{{\mathfrak H}}
\nc{\frakI}{{\mathfrak I}}
\nc{\frakJ}{{\mathfrak J}}
\nc{\frakK}{{\mathfrak K}}
\nc{\frakL}{{\mathfrak L}}
\nc{\frakM}{{\mathfrak M}}
\nc{\frakN}{{\mathfrak N}}
\nc{\frakO}{{\mathfrak O}}
\nc{\frakP}{{\mathfrak P}}
\nc{\frakQ}{{\mathfrak Q}}
\nc{\frakR}{{\mathfrak R}}
\nc{\frakS}{{\mathfrak S}}
\nc{\frakT}{{\mathfrak T}}
\nc{\frakU}{{\mathfrak U}}
\nc{\frakV}{{\mathfrak V}}
\nc{\frakW}{{\mathfrak W}}
\nc{\frakX}{{\mathfrak X}}
\nc{\frakY}{{\mathfrak Y}}
\nc{\frakZ}{{\mathfrak Z}}
\nc{\bbA}{{\mathbb A}}
\nc{\bbB}{{\mathbb B}}
\nc{\bbC}{{\mathbb C}}
\nc{\bbD}{{\mathbb D}}
\nc{\bbE}{{\mathbb E}}
\nc{\bbF}{{\mathbb F}} \nc{\bbf}{{\mathbf f}}
\nc{\bbG}{{\mathbb G}}
\nc{\bbH}{{\mathbb H}}
\nc{\bbI}{{\mathbb I}}
\nc{\bbJ}{{\mathbb J}}
\nc{\bbK}{{\mathbb K}}
\nc{\bbL}{{\mathbb L}}
\nc{\bbM}{{\mathbb M}}
\nc{\bbN}{{\mathbb N}}
\nc{\bbO}{{\mathbb O}}
\nc{\bbP}{{\mathbb P}}
\nc{\bbQ}{{\mathbb Q}}
\nc{\bbR}{{\mathbb R}}
\nc{\bbS}{{\mathbb S}}
\nc{\bbT}{{\mathbb T}}
\nc{\bbU}{{\mathbb U}}
\nc{\bbV}{{\mathbb V}}
\nc{\bbW}{{\mathbb W}}
\nc{\bbX}{{\mathbb X}}
\nc{\bbY}{{\mathbb Y}}
\nc{\bbZ}{{\mathbb Z}}
\nc{\calA}{{\mathcal A}}
\nc{\calB}{{\mathcal B}}
\nc{\calC}{{\mathcal C}}
\nc{\calD}{{\mathcal D}}
\nc{\calE}{{\mathcal E}}
\nc{\calF}{{\mathcal F}}
\nc{\calG}{{\mathcal G}}
\nc{\calH}{{\mathcal H}}
\nc{\calI}{{\mathcal I}}
\nc{\calJ}{{\mathcal J}}
\nc{\calK}{{\mathcal K}}
\nc{\calL}{{\mathcal L}}
\nc{\calM}{{\mathcal M}}
\nc{\calN}{{\mathcal N}}
\nc{\calO}{{\mathcal O}}
\nc{\calP}{{\mathcal P}}
\nc{\calQ}{{\mathcal Q}}
\nc{\calR}{{\mathcal R}}
\nc{\calS}{{\mathcal S}}
\nc{\calT}{{\mathcal T}}
\nc{\calU}{{\mathcal U}}
\nc{\calV}{{\mathcal V}}
\nc{\calW}{{\mathcal W}}
\nc{\calX}{{\mathcal X}}
\nc{\calY}{{\mathcal Y}}
\nc{\calZ}{{\mathcal Z}}

\nc{\scrA}{{\mathscr A}}
\nc{\scrB}{{\mathscr B}}
\nc{\scrC}{{\mathscr C}}
\nc{\scrD}{{\mathscr D}}
\nc{\scrE}{{\mathscr E}}
\nc{\scrF}{{\mathscr F}}
\nc{\scrG}{{\mathscr G}}
\nc{\scrH}{{\mathscr H}}
\nc{\scrI}{{\mathscr J}}
\nc{\scrJ}{{\mathscr I}}
\nc{\scrK}{{\mathscr K}}
\nc{\scrL}{{\mathscr L}}
\nc{\scrM}{{\mathscr M}}
\nc{\scrN}{{\mathscr N}}
\nc{\scrO}{{\mathscr O}}
\nc{\scrP}{{\mathscr P}}
\nc{\scrQ}{{\mathscr Q}}
\nc{\scrR}{{\mathscr R}}

\nc{\bnu}{{\bar{ \nu}}}
\nc{\olO}{\bar{\calO}}

\nc{\al}{{\alpha}} 
\nc{\be}{{\beta}}
\nc{\ga}{{\gamma}} \nc{\Ga}{{\Gamma}}
\nc{\hGa}{\hat{\Gamma}}
\nc{\ve}{{\varepsilon}} 
\nc{\la}{{\lambda}} \nc{\La}{{\Lambda}}
\nc{\om}{\omega} \nc{\Om}{\Omega} 
\nc{\sig}{{\sigma}} \nc{\Sig}{{\Sigma}}
\nc{\dR}{{\mathrm{dR}}}
\nc{\Perf}{{\mathrm{Perf}}}
\nc{\Gm}{{\mathbb{G}_m}}

\nc{\Spa}{\on{{Spa}}}
\nc{\Spd}{\on{{Spd}}}
\nc{\tnb}{\psi_{\rm tame}}
\nc{\oM}{\overline{{M}}}
\nc{\op}{{\on{op}}}
\nc{\ad}{{\on{ad}}}
\nc{\alg}{{\on{alg}}}
\nc{\Ad}{{\on{Ad}}}
\nc{\Adm}{{\on{Adm}}} \nc{\aff}{{\on{af}}}
\nc{\Aut}{{\on{Aut}}}
\nc{\Bun}{{\on{Bun}}}
\nc{\cha}{{\on{char}}}
\nc{\der}{{\on{der}}}
\nc{\Der}{{\on{Der}}}
\nc{\diag}{{\on{diag}}}
\nc{\End}{{\on{End}}}
\nc{\Fl}{{\calF\!\ell}}
\nc{\Tr}{{\on{Transp}}}
\nc{\TR}{{\calT\!\calR}}
\nc{\Gal}{{\on{Gal}}}
\nc{\Gr}{{\on{Gr}}}
\nc{\Hk}{{\on{Hk}}}
\nc{\rH}{{\on{H}}}
\nc{\Hom}{{\on{Hom}}}
\nc{\IC}{{\on{IC}}}
\nc{\id}{{\on{id}}}
\nc{\Id}{{\on{Id}}}
\nc{\ind}{{\on{ind}}}
\nc{\Ind}{{\on{Ind}}}
\nc{\Lie}{{\on{Lie}}}
\nc{\Pic}{{\on{Pic}}}
\nc{\pr}{{\on{pr}}}
\nc{\Res}{{\on{R}}}
\nc{\res}{{\on{res}}} \nc{\Sat}{{\on{Sat}}}
\nc{\spc}{{\on{sc}}}
\nc{\drv}{{\on{der}}}
\nc{\sgn}{{\on{sgn}}}
\nc{\Spec}{{\on{Spec}}}\nc{\Spf}{\on{Spf}} 
\nc{\Sph}{\on{Sph}}
\nc{\St}{{\on{St}}}
\nc{\tr}{{\on{tr}}}
\nc{\Mod}{{\mathrm{-Mod}}}
\nc{\Hilb}{{\on{Hilb}}} 
\nc{\Ext}{{\on{Ext}}} 
\nc{\vs}{{\on{Vec}}}
\nc{\ev}{{\on{ev}}}
\nc{\nO}{{\breve{\calO}}}
\nc{\tS}{{\tilde{S}}}
\nc{\spe}{{\on{sp}}}
\nc{\loc}{{\on{loc}}}

\nc{\co}{\colon}
\nc{\dia}{{\diamondsuit}}

\nc{\nscrR}{{\mathscr{R}^{\on{nr}}}}

\nc{\GL}{{\on{GL}}}

\nc{\Gl}{\on{Gl}} 
\nc{\GSp}{{\on{GSp}}}
\nc{\gl}{{\frakg\frakl}}
\nc{\SL}{{\on{SL}}} 
\nc{\SU}{{\on{SU}}} 
\nc{\SO}{{\on{SO}}}
\nc{\PGL}{{\on{PGL}}}

\nc{\Conv}{{\on{Conv}}}
\nc{\Rep}{{\on{Rep}}}
\nc{\Dom}{{\on{Dom}}}
\nc{\red}{{\on{red}}}
\nc{\act}{{\on{act}}}
\nc{\nr}{{\on{nr}}}
\nc{\ctf}{{\on{ctf}}}

\nc{\str}{{\on{-}}} 
\nc{\os}{{\bar{s}}}
\nc{\oeta}{{\bar{\eta}}}

\nc{\hookto}{\hookrightarrow}
\nc{\longto}{\longrightarrow}
\nc{\leftto}{\leftarrow}
\nc{\onto}{\twoheadrightarrow}
\nc{\lonto}{\twoheadleftarrow}

\nc{\pot}[1]{ [\hspace{-0,5mm}[ {#1} ]\hspace{-0,5mm}] }
\nc{\rpot}[1]{ (\hspace{-0,7mm}( {#1} )\hspace{-0,7mm}) }

\numberwithin{equation}{section}

\begin{document}
	
	\title{Distributions and normality theorems}
	
	\author{Jo\~ao Louren\c{c}o}

	\address{Mathematisches Institut, Universität Münster, Einsteinstrasse 62, Münster, Germany}
	\email{j.lourenco@uni-muenster.de}
	
	\maketitle
	
	\begin{abstract}
We derive a Serre presentation of distribution algebras of loop groups in characteristic $p$ and apply it to give a new proof of the normality of Schubert varieties inside parahoric affine Grassmannians, for all connected reductive groups whose fundamental group is $p$-torsion free. 
	\end{abstract}

\tableofcontents

\section{Introduction}

Let $F$ be a field of characteristic $p$ and $G$ be a connected reductive group over $F$. If we want to understand the infinitesimal behavior of $G$ near the identity, it is well known from modular representation theory that $\mathrm{Lie}(G)$ is insufficient. Indeed, the category of $\mathrm{Lie}(G)$-modules barely captures information on the category of $G$-representations. Instead, what one ought to consider is the $F$-algebra $\mathrm{Dist}(G)$ of distributions of $G$ at the origin consisting of higher differential operators. This object is a sort of twisted divided power algebra and an explicit presentation in terms of generators and relations was given by Takeuchi \cite{Tak83a,Tak83b} for split $G$. The idea for producing these generators and relations is to take the distributions of tori and root groups and evaluating on the rules of multiplications between those subgroups of $G$. In particular, the list of relations is not finite, and is best expressed in terms of generating series. 

The point of $\mathrm{Dist}(G)$ is that it carries the same information as the formal group $\hat{G}$ given as the completion of $G$ at the origin but in such a way that we get a covariant algebra object instead of a covariant geometric space (or contravariant algebra object by passing to formal sections). We became interested in it because of how it naturally fits into the study of loop groups. Assume from now on that $F=k\rpot{t}$ is a local field with finite residue field $k$. The loop group of $G$ is the group object in ind-schemes $\Res_{F/k}G$ whose functor of points is given by $R\mapsto G(R\rpot{t})$. Similarly, if we let $\calG$ be a parahoric model of $G$ over the ring of integers $O=k\pot{t}$ in the sense of Bruhat--Tits \cite{BT72,BT84}, then we have the group object $\Res_{O/k}\calG$ in schemes given by $R \mapsto \calG(R\pot{t})$. The affine flag variety $\mathrm{Gr}_\calG$ is the étale quotient $\Res_{F/k}G/\Res_{O/k}\calG$ in the category of ind-schemes and classifies modifications of $\calG$-torsors over $O$ away from the residue field $k$. The Bruhat stratification yields Schubert varieties $\mathrm{Gr}_{\calG,l,\leq w}$ defined over the reflex field $l$ of the element $w$ in the absolute Iwahori--Weyl group. 

\begin{theorem}\label{teo1_intro}
	If $p\nmid \# \pi_1(G_{\mathrm{der}})$, then $\mathrm{Gr}_{\calG,l,\leq w}$ is normal, Cohen--Macaulay, rational and globally $+$-regular.
\end{theorem} 

This result appears already in \cite{Fal03,PR08,FHLR22} for every group except odd unitary ones if $p=2$.
The point of this paper is to give a new uniform proof, so first we have to review the history behind it. Faltings \cite{Fal03} proved the theorem for split groups and his proof had two steps: (i) applying the Mehta--Ramanathan criterion \cite{MR85} on $\varphi$-split varieties, where $\varphi$ denotes the Frobenius map, to show that the transition maps of normalized Schubert varieties are closed immersions; (ii) use integral Lie-theoretic arguments to prove that the embedded Schubert varieties are already normal. Then, Pappas--Rapoport \cite{PR08} applied this strategy to the case of tame groups. For non-tame $G$, we face the following obstacles: (i) demands a divisor on $\mathrm{Gr}_\calG$ with specific degrees, usually defined via negative loop groups, which do not seem to exist beyond the tame case; (ii) requires an integral lift of the group theoretic data to a two dimensional ring such as $W(k)\pot{t}$, which does not seem to exist for odd unitary groups if $p=2$. 

A non-negligible portion of our research was dedicated to overcome some of the above obstacles. First, we show in \cite{HLR18} that most Schubert varieties are not normal if $p \mid \#\pi_1(G_{\mathrm{der}})$, see \cite{BR23} for the full classification at hyperspecial level. As for tameness, we lifted it almost completely in \cite{Lou23,FHLR22} as follows. For part (i), we retrieve the critical divisors for all groups by a reduction to the split case, still relying on a case-by-case analysis. For part (ii), we constructed $W(k)$-lifts for all groups except odd unitary ones when $p=2$, and could finish Faltings' proof of embedded normality. 

Since then, we found an alternative approach to (i). In \cite{CL24} we replace the Mehta--Ramanathan criterion \cite{MR85} for $\varphi$-splitness by the Bhatt criterion \cite{Bha12} for splinters, also known as globally $+$-regular varieties in \cite{BMP+23}. This criterion involves choosing auxiliary $\bbQ$-Cartier boundaries inside Demazure varieties but does not require defining Cartier divisors in $\mathrm{Gr}_\calG$. This gives a uniform proof of part (i) for all $G$. Improving the strategy for part (ii) is the goal of this paper: we want to explain a new method that uniformly proves embedded normality for all groups.

At this point, the reader probably already guessed how distributions fit into the normality picture. The main idea is that the integral Lie algebra methods should be replaced by dealing directly with the associative $k$-algebra $\mathrm{Dist}(\Res_{F/k}G)$ of loop distributions. This algebra carries a natural topological structure, so we use the formalism of condensed mathematics of Clausen--Scholze \cite{CS19}. We define a Serre presentation for $\mathrm{Dist}(\Res_{F/k}G)$ following \cite{Tak83a,Tak83b} to reduce to rank $1$ groups. If $G$ is a restriction of scalars of $\mathrm{SL}_2$, we can make the calculations explicitly, as its Schubert varieties are lci. If $G$ is a restriction of scalars of $\mathrm{SU}_3$, we couldn't perform the calculations effectively, so instead more effort is required via the theory of local models, which we explain next.

Beilinson--Drinfeld \cite{BD91} attach a more general loop group $\Res_{O^2_\circ/O} \calG$ to the parahoric model $\calG$, where the restriction of scalars is along the inclusion of the second factor in the complement of the diagonal of $O$, and the $\calG$-structure is along the first factor. It carries the jet subgroup $\Res_{O^2/O} \calG$ and the étale quotient $\mathrm{Gr}_{\calG,O}$ is called the Beilinson--Drinfeld Grassmannian. Its generic fiber is isomorphic to the usual affine Grassmannian $\mathrm{Gr}_{G,F}$ of the $F$-group $G$, whereas the special fiber equals $\mathrm{Gr}_{\calG,k}$. Let $\mathrm{Gr}_{\calG,O_E,\leq \mu}$ be the scheme-theoretic image of $\mathrm{Gr}_{G,E,\leq \mu}\to \mathrm{Gr}_{\calG,O_E}$, where $\mu$ is a conjugacy class of absolute coweights of $G$ and $E$ is its reflex field.

\begin{theorem}\label{teo2_intro}
	If $p\nmid \#\pi_1(G_{\mathrm{der}})$, then $\mathrm{Gr}_{\calG,O_E,\leq \mu}$ is normal and Cohen--Macaulay with $\varphi$-split special fiber.
\end{theorem}

Again the proof of \Cref{teo2_intro} can be found in \cite{Zhu14,FHLR22} if $p>2$ or $G$ is $\mathrm{SU}_3$-free. Zhu \cite{Zhu14} handles tamely ramified groups via a global $\varphi$-splitting, which we cannot generalize to all $G$. Instead, we prove it via the unibranch theorem of \cite{GL22} together with \Cref{teo1_intro}. We also use $\Gr_{\calG,O}$ to finish the proof of \Cref{teo1_intro} when $G=\mathrm{SU}_3$. If $\calG$ is a special parahoric with simply connected reductive quotient, then $\mathrm{Dist}(\Res_{O^2_\circ/O}\calG)$ is generated by its unipotent part. Using the case of $\mathrm{SL}_3$ in the generic fiber, this implies semi-normality of $\Gr_{\calG,O}$ and hence also of its special fiber $\mathrm{Gr}_{\calG,k}$ by a computation with minuscule coweights.

\subsection{Acknowledgements}
This project was carried out during stays at the Max-Planck-Institut für Mathematik and the Universität Münster. It received financial support from the the Excellence Cluster of the Universität Münster, and the ERC Consolidator Grant 770936 via Eva Viehmann.
I'd like to thank Johannes Anschütz, Robert Cass, Najmuddin Fakhruddin, Ian Gleason, Tom Haines, and Timo Richarz, for their tremendous help as collaborators in related projects. Thanks to Peter Scholze for various helpful discussions over the years. 

\section{Affine Grassmannians}

\subsection{Ind-schemes}
In this paper, an ind-scheme will always mean a colimit of qcqs schemes along closed immersions. The qcqs hypothesis facilitates handling morphisms of ind-schemes, as they necessarily respect scheme presentations. Qcqs formal schemes obviously embed fully faithfully into the category of ind-schemes. Indeed, it is important for us to regard formal schemes as ind-schemes, and thus their underlying reduced subscheme equals the reduction in the category of ind-schemes. 

Let $k$ be a finite field and consider the category of pointed ind-$k$-schemes $(X,x)$, where $X$ is an ind-scheme over $k$ and $x$ is a $k$-valued point of $X$.
We say that a pointed $k$-scheme $(Z,z)$ is nilpotent if its reduction equals $z$ (and thus $Z$ is affine) and its radical ideal $I_z$ is nilpotent, i.e., a power $I_z^n$ of it vanishes for some $n\gg 0$. It is decisive to focus our attention on nilpotent rather than nil-ideals, i.e., those whose elements are nilpotent, because we will encounter many ideals which are not finitely generated. 

\begin{definition}\label{def_formal_completion}
	The formal completion $\widehat{X}_x$ of a pointed ind-$k$-scheme $(X,x)$ is the filtered colimit of all closed nilpotent pointed $k$-subschemes $(Z,x)\subset (X,x)$. 
\end{definition}

Similarly, we define the ring of formal sections $\Gamma(\widehat{X}_x,\calO)$ of $(X,x)$ to be the limit of the rings $\Gamma(Z,\calO)$. This ring admits the structure of a solid commutative $k$-algebra in the sense of Clausen--Scholze \cite[Proposition 7.5, Theorem 8.1]{CS19}, as $\Gamma(Z,\calO)$ is a colimit of finite $k$-modules and $k$ is a finitely generated $\bbZ$-algebra. 
Moreover, this induces a contravariant functor from pointed ind-$k$-schemes to solid commutative $k$-algebras.

We shall also apply the notion of formally unramified, formally étale, and formally smooth maps $f\colon (X,x)\to (Y,y)$. For us, this means that 
\begin{equation}\mathrm{Hom}((Z,z),(X,x))\to \mathrm{Hom}((W,w),(X,x)) \times_{\mathrm{Hom}((W,w),(Y,y))} \mathrm{Hom}((Z,z),(Y,y)) 
\end{equation}
is either injective, bijective, or surjective, for every closed embedding $(W,w)\subset (Z,z)$ of nilpotent pointed $k$-schemes. 
During the main part of the paper, we will encounter formally étale maps of ind-schemes which are far from being representable, so the following assertion will be key.
\begin{lemma}\label{lem_form_etale_completion}
	A formally étale map $f\colon(X,x)\to (Y,y)$ of pointed ind-$k$-schemes induces an isomorphism on formal completions.
\end{lemma}

\begin{proof}
	Without loss of generality, we may and do assume that $Y$ is a nilpotent scheme and $X$ is its own formal completion at $x$. By formal étaleness, we get a unique section $s\colon Y\to X$ of $f$. This factors through a nilpotent subscheme $X'$ by quasi-compactness and the resulting map $f'\colon X'\to Y$ is necessarily formally étale. In particular, we may now assume that $X=X'$ is a nilpotent $k$-scheme. Finally, we claim that also $s\circ f$ is the identity map of $X$. By formal étaleness, we can check this after post-composing with $f$, and it is then obvious.
	\end{proof}

Similarly, formal étaleness can be detected in terms of formal sections.
\begin{lemma}\label{lem_form_etale_local_rings}
	A map $f\colon(X,x)\to (Y,y)$ of pointed ind-schemes is formally étale if and only if it induces an isomorphism of their formal sections as solid commutative $k$-algebras.
\end{lemma}
\begin{proof}
	We may and do assume that $Y$ is a nilpotent scheme and $X$ is a nilpotent ind-scheme. Then, we have that $\Gamma(Y,\calO)=\Gamma(X,\calO)$ where the left side is discrete (i.e., a colimit of finite $k$-modules) and the right side is a limit of discrete solid $k$-modules. In particular, the map factors through some nilpotent subscheme $X'\subset X$. We deduce that $\Gamma(X',\calO)=\Gamma(X,\calO)$ and hence $X'=X$.
\end{proof}

We are also interested in understanding absolute properties for ind-schemes. We will say that an ind-scheme has a certain property $(P)$ if it admits a presentation all of whose constituents satisfy the property $(P)$, compare with \cite[Definition 1.15]{Ric20}. 
If $(P)$ is preserved under closed immersions, then it does not depend on the choice of a presentation: this holds for many of the adjectives that we will employ for ind-schemes such as affine, separated, proper, projective. Then, there are other properties that one can also define via a universal property, such as reduced and semi-normal. It turns out that there always exist a universal reduced sub-ind-scheme $X^{\mathrm{red}}\to X$, see \cite[Lemma 1.17]{Ric20}. 

Let us review the notion of semi-normal schemes. A scheme $X$ is semi-normal if every universal homeomorphism $Y\to X$ with trivial residue field extensions is an isomorphism. For any scheme $X$, there exists a initial morphism $X^{\mathrm{sn}}\to X$ with $X^{\mathrm{sn}}$ semi-normal: we call it the semi-normalization of $X$. The assignment $X \to X^{\mathrm{sn}}$ defines a functor, i.e., morphisms lift to their semi-normalizations. We refer to \cite[Tag 0EUS]{StaProj} for the proof of the previous assertions. Taking colimits, we see that $X\mapsto X^{\mathrm{sn}}$ extends to a functor from ind-schemes to sheaves on $k$-algebras. The resulting sheaf will be an ind-scheme too if the transition maps are closed immersions. Fortunately, this will happen for loop groups of reductive and unipotent groups.

\subsection{Loop groups}
From now on, let $k$ be a finite field, $F=k\rpot{t}$ the local field of Laurent series with coefficients in $k$, and $O=k\pot{t}$ be the power series ring over $k$. In this paper, we are interested in the loop space $\Res_{F/k}X$ of a finite type affine $F$-scheme $X$. This is the functor on $k$-algebras given by $A\mapsto X(A\rpot{t})$, and it is representable by an ind-$k$-scheme, compare with \cite[1.a]{PR08}. If $\calX$ is an affine $O$-model of $X$, we can define a jet space $\Res_{O/k}\calX$ as the functor given by $A\mapsto X(A\pot{t})$. It turns out that the natural map $\Res_{O/k}\calX \to \Res_{F/k}X$ is a closed immersion and it realises the left side as an affine scheme that is almost never of finite type, compare again with \cite[1a]{PR08}. Let us mention some of their basic properties, starting with jet spaces.

\begin{lemma}\label{lem_res_jet}
	The functor $\calX \mapsto \Res_{O/k}\calX$ from finite type affine $O$-schemes to  affine $k$-schemes is a limit of the functors $\Res_{O_n/k}$, where $O_n=O/t^nO$. In particular, $\Res_{O/k}$ preserves immersions and carries smooth schemes to pro-smooth schemes.
\end{lemma}
\begin{proof}
	This is standard, compare with \cite[1a]{PR08}, \cite[Proposition 1.3.2]{Zhu16} and \cite[Lemma 3.17]{Ric20}. The first part follows by definition itself of the jet group and affineness of $\calX$ for commuting with limits. It is clear that it preserves closed immersions by general properties of limits along affine morphisms, and also open immersions because $k$ is the reduction of $O_n$, which implies that $\Res_{O_n/k}\calU\subset \Res_{O_n/k}\calX$ is the pullback of $\calU_k \subset \calX_k$. The final assertion on smooth schemes is obvious by \cite[Proposition A.5.11]{CGP15}.
\end{proof}

For the loop group functor $\Res_{F/k}$, it is no longer true that it respects open immersions. Indeed, we will see that $\Res_{F/k}\bbG_{\mathrm{m},F}$ is not reduced in \Cref{prop_reduced_loop_group}, whereas $\Res_{F/k}\bbG_{\mathrm{a},F}$ clearly is. However, this still has a remedy at the formal level.

\begin{lemma}\label{lem_res_loop}
	The functor $X\mapsto \Res_{F/k}X$ from finite type affine $F$-schemes to affine ind-$k$-schemes preserves closed immersions and formally étale maps.
\end{lemma}

\begin{proof}
	The assertion for closed immersions follows easily by reduction to $X=\bbA^n_F$. 
	As for formal étaleness, let $(Z,z)$ be any nilpotent scheme. Showing that a map $\Res_{F/k}f\colon \Res_{F/k}X\to \Res_{F/k}Y$ lifts uniquely against the inclusion $z \subset Z$ amounts to showing that $f\colon X \to Y$ lifts uniquely against $z\rpot{t}\subset Z\rpot{t}$, which holds by definition. Here, $(Z\rpot{t},z\rpot{t})$ indicates the nilpotent pointed $F$-scheme such that $\Gamma(Z\rpot{t},\calO_{Z\rpot{t}})=\Gamma(Z,\calO_Z)\rpot{t}$.
\end{proof}


Let $G$ be a connected reductive $F$-group and $\calG$ be a parahoric $O$-model of $G$ in the sense of \cite[5.1.9, Définition 5.2.6]{BT84}. Note that the loop and jet spaces $\Res_{F/k}G$ and $\Res_{O/k}\calG$ are now group objects in the category of ind-schemes. Now, we turn to the flag variety arising as the quotient \begin{equation}\mathrm{Gr}_\calG=(\Res_{F/k}G)/(\Res_{O/k}\calG)\end{equation} for the étale topology. This is representable by an ind-scheme by \cite[Theorem 1.4]{PR08}. 
The Bruhat decomposition for parahoric models of reductive groups over non-archimedean fields assumes the form
\begin{equation}
	G(\breve{F})= \bigsqcup_{w \in W_\calG\backslash W /W_\calG} \calG(\breve{O})\dot{w} \calG(\breve{O}).
\end{equation}
Here, we let $T \subset G$ denote a maximally $\breve{F}$-split maximal $F$-torus, see \cite[Corollaire 5.1.12]{BT84}, $\calT$ be the connected Néron $O$-model of $T$, $N$ the normalizer of $T$ inside $G$, $W:=N(\breve{F})/\calT(\breve{O})$ the Iwahori--Weyl group, and $W_\calG \subset W$ the subgroup of elements whose lifts $\dot{w} \in N(\breve{F})$ are contained in $\calG(\breve{O})$, see \cite[Theorem 7.5.3, Proposition 7.5.5]{KP23}.
For every $w \in W$, we define $\mathrm{Gr}_{\calG,l,w}$ as the étale descent to the reflex field $l$ of $w$ of the smooth locally closed orbit of $w$ inside $\Gr_{\calG,\bar k}$ under the jet $\bar{k}$-group of $\calG$. Its scheme-theoretic closure $\mathrm{Gr}_{\calG,l,\leq w}$ inside $\mathrm{Gr}_{\calG,l}$ is a reduced integral projective scheme, compare with \cite[Definition 8.3]{PR08}. This is called the Schubert variety associated with $w$ and our goal in this paper is to study its geometry.

We denote by $\mathrm{Gr}_{\calG,l,\leq w}^{\mathrm{sn}}$ the semi-normalizations of the Schubert varieties embedded in $\mathrm{Gr}_{\calG,l}$. Note that, while we have canonical transition morphisms $\mathrm{Gr}^{\mathrm{sn}}_{\calG,l,\leq v} \to \mathrm{Gr}^{\mathrm{sn}}_{\calG,l,\leq w}$ over a common reflex field, it is not clear whether they are closed immersions. In particular, it is not a priori clear that the sheaf $\Gr_{\calG}^{\mathrm{sn}}$ is an ind-scheme in our strict sense. Fortunately, these problems dissipate thanks to the following key theorem on semi-normalized Schubert varieties.

\begin{theorem}[\cite{FHLR22,CL24}]\label{thm_splinter}
	The $l$-varieties $\mathrm{Gr}_{\calG,l,\leq w}^{\mathrm{sn}}$ are normal, Cohen--Macaulay, rational, compatibly $\varphi$-split and globally $+$-regular. In particular, the semi-normalization $\mathrm{Gr}_{\calG}^{\mathrm{sn}}$ is an ind-$k$-scheme.
\end{theorem}

As explained in the introduction, this result goes back to \cite[Theorem 8]{Fal03} for split $G$, \cite[Theorem 8.4]{PR08} for tame $G$, and \cite[Theorem 4.1]{FHLR22} for all $G$. A new uniform proof was given in \cite{CL24}.
Before concluding this section, we want to explain the several notions appearing in the statement and the proof strategies.

A noetherian scheme $X$ is Cohen--Macaulay if the depth of its local rings (i.e., the maximal length of its regular sequences) equals their Krull dimensions, see \cite[Tag 00N7]{StaProj}. Equivalently, one can define Cohen--Macaulayness by demanding that the dualizing complex $\omega_X^\bullet$ is concentrated in a single degree or that the lower (i.e., below the Krull dimension) local cohomology groups $H^i_x(\calO_{X,x})$ all vanish, compare with \cite[Definition 2.1, Example 2.5]{Bha20}. We say following Kovács \cite[Definition 1.3]{Kov17} that a normal Cohen--Macaulay $k$-variety $X$ is rational if for any proper birational map $f\colon Y\to X$ with $Y$ also normal and Cohen--Macaulay, the higher direct image sheaves $R^if_*\calO_Y$ vanish for all $i>0$. It is enough to verify this condition for a resolution of $X$ by \cite[Theorem 9.12]{Kov17} and it implies vanishing of $R^if_*\calO_Y$ for all $i>0$ by \cite[Theorem 8.6]{Kov17}. Finally, we say that $X$ is $\varphi$-split if the Frobenius $\varphi\colon \calO_X \to \varphi_*\calO_X$ splits as a map of $\calO_X$-modules, and compatibility carries the obvious meaning, compare with \cite[Definitions 5.0.1 and 5.1.4]{BS13}. More generally, $X$ is a splinter in the sense of \cite[Definition 0.1]{Bha12} or globally $+$-regular in the sense of \cite[Definition 6.1]{BMP+23} if $\calO_X \to \calO_Y$ splits as $\calO_X$-modules for any finite cover $Y \to X$. It is proved in \cite[Corollary 5.3]{Bha12} that splinters are automatically normal and Cohen--Macaulay. There is also a notion of strong $\varphi$-regularity generalizing $\varphi$-splitness and which is only slightly stronger than global $+$-regularity.

Let us now sketch the proof of \Cref{thm_splinter}. Let $\calI$ be a Iwahori $O$-model of $G$ obtained from $\calG$ by dilatation. The transition map $\Gr_{\calI} \to \Gr_{\calG}$ is proper smooth, so it suffices to handle the Iwahori case. Up to translation into the neutral component and after enlarging the finite field $k$, any Schubert variety $\Gr_{\calI,\leq w}$ is resolved by the convolution Schubert variety $\Gr_{\calI,\leq s_\bullet}$ where $s_\bullet$ is a sequence of positive simple reflections in the absolute Iwahori--Weyl group. This variety is an iterated $\bbP^1_k$-bundle, so it is smooth in particular, and it is usually called the BSDH variety after work of Bott--Samelson \cite{BS58}, Demazure \cite{Dem74}, and Hansen \cite{Han73}. Prior to \cite{CL24}, the existence of a $\varphi$-splitting of $\Gr_{\calI,\leq s_\bullet}$ was always deduced via the Mehta--Ramanathan criterion \cite{MR85}. This requires producing a Cartier divisor on $\Gr_{\calI}$ with degree $1$ on each $\Gr_{\calI,\leq s_i}$, which is difficult to define homogeneously, see \cite[Section 4]{FHLR22} for the full proof using some case division. This upgrades to a splitting of the absolute integral closure by the proof of \cite[Theorem 1.4]{Cas22}, whose method goes back to \cite{LRPT06} for classical flag varieties. In \cite{CL24}, we approach the problem instead via inversion of adjunction following \cite[Proposition 7.2]{Bha12} refined as in \cite[Theorem 7.2]{BMP+23}, i.e.~we split absolute integral closure by induction on the length of $s_\bullet$. We have to perform a few calculations with auxilliary boundaries, but the proof remains uniform for all $G$ throughout.

\subsection{A rank $1$ example}
In this subsection, we calculate Schubert varieties for rank $1$ split groups, and in particular verify their normality as in \Cref{teo_seminormal_schubert}. We will do this by exploiting certain presentations of Schubert varieties at special level, which appear in \cite[Subsection 1.2.2]{Zhu17} and also \cite[Lemmas 19.3.5 to 19.3.7]{SW20}.
Let $G=\mathrm{GL}_{2,F}$, $\calG=\mathrm{GL}_{2,O}$, and $\mu=(1,0)$ be its only dominant minuscule coweight. Up to translation, the Schubert varieties inside $\mathrm{Gr}_{\mathrm{GL}_{2,O}}$ are of the form $\mathrm{Gr}_{\mathrm{GL}_{2,O},\leq n\mu}$.

\begin{proposition}\label{prop_lci}
	The variety $\mathrm{Gr}_{\mathrm{GL}_{2,O},\leq n\mu}$ is a normal lci.
\end{proposition}

\begin{proof}
	We follow \cite[Lemma B.4]{Zhu17}. Note that every $\Res_{O/k}\mathrm{GL}_{2,O}$-orbit has codimension at least $2$ in a larger one, so we at least know that every Schubert variety is regular in codimension $1$. We just have to verify that it is lci. 
	
	Let $\mathrm{Mat}_{2,O}$ be the $O$-scheme of $2$-by-$2$ matrices and consider the closed $k$-subscheme $X_n\subset \Res_{O/k}\mathrm{Mat}_{2,O}\cap \Res_{F/k}\mathrm{GL}_{2,F}$ given by the preimage of $t^n$ along the determinant map $\mathrm{det}\colon \Res_{F/k}\mathrm{GL}_{2,F} \to \mathrm{Gr}_{\bbG_{\mathrm{m},O}}$. We claim that the map $X_n \to \mathrm{Gr}_{\mathrm{GL}_{2,O}}$ equals the natural $\Res_{O/k}\mathrm{GL}_{2,O}$-torsor over $\mathrm{Gr}_{\mathrm{GL}_{2,O},\leq n\mu}$.
	
	At the level of geometric points, this equality of closed subfunctors is an immediate consequence of the elementary divisor theorem. Therefore, it suffices to verify that $X_n$ is reduced. For this, we note that $X_n$ is a torsor over a finite type $k$-scheme $Y_n$ under the congruence subgroup of $\Res_{O/k}\mathrm{Mat}_{2,O}$ arising as the kernel of the map towards $\Res_{O_{n+1}/k}\mathrm{Mat}_{2,O_{n+1}}$, where $O_{n+1}=O/t^{n+1}O$. We start by proving that $Y_n$ is generically reduced. For this, consider its tangent space at the matrix $(t^n,0,0,1)$. It equals the $k$-submodule of $\mathrm{Mat}_2(O_{n+1})$ whose $(1,1)$-entry is divisible by $t^n$, and thus has dimension equal to $4(n+1)-n=3n+4$. If we consider the stabilizer of the matrix $(t^n,0,0,1)$ under left and right multiplication by $\Res_{O_{n+1}/k}\mathrm{GL}_{2,O_{n+1}}$, we get the matrix equality $(t^na, b, t^nc, d)=(t^ne,t^nf,g,h)$, and hence the topologically dense smooth orbit has dimension equal to $8n+8-4n-4-n=3n+4$, and this implies generic reducedness.
	Finally, if we write down the determinant as a power series modulo $t^{n+1}$, we conclude that $Y_n$ sits inside $\bbA_k^{4n+5}$ and is defined by $n+1$ equations.  This implies that $Y_n$ is indeed a normal complete intersection.
\end{proof}

\subsection{Normality}
In this subsection, we discuss reducedness and semi-normality of loop groups $\Res_{F/k}G$ attached to connected reductive groups. For simplicity, we use the shorthand notation $\Res_{F/k}^{\mathrm{red}}G:=(\Res_{F/k}G)^{\mathrm{red}}$, resp.~$\Res_{F/k}^{\mathrm{sn}}G:=(\Res_{F/k}G)^{\mathrm{sn}}$ to denoted the reduction and semi-normalization of our ind-schemes. Let us start by the problem of reducedness, since it is the most simple.

\begin{proposition}\label{prop_reduced_loop_group}
	The ind-scheme $\Res_{F/k}G$ is reduced if and only if $G$ is semi-simple and $p \nmid \#\pi_1(G)$.
\end{proposition}
\begin{proof}
	The non-reducedness for non-semisimple $G$ was proved in \cite[Proposition 6.5]{PR08}
	The obstruction for tori $T$ lies in the fact that the reduction of $\mathrm{Gr}_{T}$ is a zero-dimensional scheme locally of finite type, and hence $\mathrm{Lie}(\Res_{F/k}^{\mathrm{red}}T)=\mathrm{Lie}(\calT)$ is not equal to $\mathrm{Lie}(T)$. Let $A=G/G_{\mathrm{der}}$ denote the abelian quotient of $G$. 
	If $\Res_{F/k}G$ were reduced, then $\mathrm{Lie}(G)$ would map onto $\mathrm{Lie}(\calA) \subset \mathrm{Lie}(A)$, contradicting smoothness of the map of $F$-schemes $G\to A$.
	
	Assume that $G$ is semi-simple but $p$ divides the order of $\pi_1(G_{\mathrm{der}})$. In \cite[Proposition 7.10]{HLR18}, we proved that $\Res_{F/k}G$ is non-reduced when $G$ is tame. First, note that $G_{\mathrm{sc}}\to G$ is a central isogeny whose kernel $\mu$ is not étale. In particular, we know by \cite[Examples 1.3.2 and A.7.9]{CGP15} that $\Res_{F/F^p}\mu$ is positive-dimensional and hence $\Res_{F/F^p}G_{\mathrm{sc}}$ does not surject onto $\Res_{F/F^p}G$. Since Schubert varieties of isogenous group have isomorphic semi-normalizations, we deduce that if $\Res_{F/k}G$ were reduced, then we would have an equality \begin{equation}\mathrm{Lie}(\Res_{F/F^p}G_{\mathrm{sc}}/\Res_{F/F^p}\mu_p)+\mathrm{Lie}(\calG) =\mathrm{Lie}(G)\end{equation}
	of Lie algebras. But the left side is the sum of a proper $F^p$-subspace and an $O^p$-lattice, so this equality can never hold.
	
	Next, we consider the case when $G$ is simply connected, that was handled by \cite[Corollary 11]{Fal03} for split $G$ and \cite[Proposition 9.9]{PR08} for tame $G$. It will follow from our calculations of distribution algebras later on that $\Res_{F/k}^{\mathrm{red}}G \to \Res_{F/k}G$ is formally étale, e.g., see \Cref{lem_dist_iso_form_etale} and \Cref{prop_surj_dist_pm_simple_full}. Modding out by the pro-smooth group $\Res_{O/k}\calG$, we deduce that the reduction map of $\Gr_{\calG}$ is formally étale at every point by homogeneity. Now, \cite[Lemma 8.6]{HLR18} implies that it is an isomorphism. 
	
	Finally, we treat the case when $G$ is semisimple and $p$ does not divide the order of $\pi_1(G_{\mathrm{der}})$, due to \cite[Theorem 6.1]{PR08} for tame $G$. Notice that the kernel $\mu$ of $G_{\mathrm{sc}}\to G$ is étale. Consequently, $1_k=\Res_{F/k}1_F\to \Res_{F/k}\mu$ is formally étale, and the right side has finitely many points, so we conclude that it is also an étale $k$-scheme. In particular, $\Res_{F/k}G_{\mathrm{sc}}\to \Res_{F/k}G$ is an étale cover on neutral components and we deduce that $\Res_{F/k}G$ is also reduced.
\end{proof}

Next, we move to the problem of semi-normality. In this paper, we work mostly with the full loop group $\Res_{F/k}G$ during the proofs to fully exploit its multiplication law. However, for classical reasons, we state the result below for Schubert varieties in $\Gr_{\calG}$.
\begin{theorem}\label{teo_seminormal_schubert}
	If $p\nmid \#\pi_1(G_{\mathrm{der}})$, then $\Gr_{\calG,l,\leq w}$ is normal.
\end{theorem}

In particular, we deduce by \Cref{thm_splinter} that the Schubert varieties $\mathrm{Gr}_{\calG,\leq w}$ are normal, Cohen--Macaulay, rational, and globally $+$-regular.
This result is found in \cite[Theorem 8]{Fal03} for split $G$, \cite[Theorem 8.4]{PR08} for tame $G$, and \cite[Theorem 4.23]{FHLR22} if $p>2$ or $G$ is $\mathrm{SU}_3$-free. As mentioned in the introduction, those papers employ Lie-theoretic considerations to ad hoc $W(k)$-lifts of $\Res_{F/k}G$. In this paper, we argue via distributions and therefore our proof must be postponed until the very end, see \Cref{cor_su3_free_seminormal,cor_seminormal_all}. Before moving on, let us perform a helpful reduction step.

\begin{lemma}\label{lem_reduction_step_form_etale}
	\Cref{teo_seminormal_schubert} holds if $\Res_{F/k}^{\mathrm{sn}}G_{\mathrm{sc}}\to \Res_{F/k}G_{\mathrm{sc}}$ is formally étale.
\end{lemma}

\begin{proof}
First of all, note that, if $p$ does not divide the order of $\pi_1(G_{\mathrm{der}})$, the neutral component of $\Res_{F/k}^{\mathrm{red}}G$ admits $\Res_{F/k}G_{\mathrm{sc}}$ as an étale cover by the proof of \Cref{prop_reduced_loop_group}, so this reduces our problem to simply connected $G$. 
Suppose that $\Res_{F/k}^{\mathrm{sn}}G\to \Res_{F/k}G$ is formally étale. Quotienting out the right action of the pro-smooth jet group $\Res_{O/k}\calG$, we deduce that $\Gr^{\mathrm{sn}}_{\calG} \to \Gr_{\calG}$ is formally étale. Restricting to Schubert varieties, we also see that $\Gr_{\calG,\leq w}^{\mathrm{sn}}\to \Gr_{\calG,\leq w}$ is formally unramified. This implies by \cite[Tag 04XV]{StaProj} that the semi-normalization map is a closed immersion, and thus an isomorphism, around the identity. By $\Res_{O/k}\calG$-equivariance, this propagates to the entire variety, which is normal by \Cref{thm_splinter}.
\end{proof}

\subsection{Beilinson--Drinfeld deformation}
In this subsection, we discuss the deformation of the $k$-ind-scheme $\Gr_{\calG}$ to the ring of integers $O$ as defined by Beilinson--Drinfeld \cite{BD91}. This means we have to consider the category of qcqs pointed ind-schemes $(X,x)$ over $O$, where $x$ stands for a section of the structure map $X \to \Spec(O)$. We can still define corresponding $O$-relative versions of the formal completion and the ring of formal sections, where the latter carries the structure of a solid commutative $O$-algebra for the $t$-adic topology, compare with \cite[Proposition 7.9]{CS19}. Note that \Cref{lem_form_etale_completion,lem_form_etale_local_rings} hold also in the $O$-relative setting with the same proofs. Similarly, we can define the various properties such as separated, proper, affine, reduced, and semi-normal. 

A novelty here is the notion of flatness over $O$ (this is automatic over a field), which gives rise to a functor $X \mapsto X^{\mathrm{fl}}$ given as the scheme-theoretic image of the generic fiber and called the flat closure, see \cite[Definition 8.3]{HLR18}. We say moreover that $(X,x)$ is formally flat if the formal completion is flat as an ind-scheme. Note that a normal flat $O$-scheme of finite type is formally flat because localizations of formal schemes are flat and normality is preserved under completion for excellent rings.

Let $\calX$ be a finite type affine $O$-scheme. We define the loop space $\Res_{O_\circ^2/O}\calX$ as the affine ind-scheme representing the functor $A \mapsto \calX(A\rpot{t-a})$ on $O$-algebras. Here, $a$ denotes the image of $t \in O$ in $A$ via the structure map and we regard $A\rpot{t-a}$ as an $O$-algebra instead via the formal variable $t$. Similarly, we have a closed affine subscheme $\Res_{O^2/O}\calX \subset \Res_{O^2_\circ/O}\calX$ called the jet space and representing the functor $A \mapsto \calX(A\pot{t-a})$ on $O$-algebras. \Cref{lem_res_jet,lem_res_loop} admit counterparts in the $O$-relative setting and we also have the following flatness result.

\begin{lemma}\label{lem_flat_bdloop}
	If $G$ is quasi-split with induced maximal torus $T$, then the ind-scheme $\Res_{O^2_\circ/O}\calG$ is formally flat.
\end{lemma}
\begin{proof}
	This is proved in \cite[Corollary 8.5, Proposition 8.8]{HLR18} for tame $G$ and without assuming $T$ to be induced. Note that $\calG$ has a big cell $\calC:=\calU^-\times \calT \times \calU^+$ as an open neighborhood of the identity by \cite[Théorème 3.8.1]{BT84}, where products are fibered over $O$. The big cell itself can be identified with an open neighborhood of the zero section in $\bbA_O^n$, where $n$ equals the dimension of $G$. In particular, by the $O$-relative variant of \Cref{lem_res_loop}, the loop spaces $\Res_{O^2_\circ/O}\calG$ and $ \Res_{O^2_\circ/O}\bbA_{O}^n$, have isomorphic formal completions, and one checks easily that the latter is formally flat.
\end{proof}

For a parahoric $O$-model $\calG$ of a given connected reductive $F$-group $G$, we can define the affine Grassmannian
\begin{equation}
	\Gr_{\calG,O}:=\Res_{O^2_\circ/O}\calG/\Res_{O^2/O}\calG
\end{equation}
as the quotient for the étale topology and it is representable by a projective ind-scheme by \cite[Proposition 5.5]{PZ13}. Its generic fiber $\Gr_{G,F}$ is the usual affine Grassmannian attached to the $F$-group $G$, whereas the special fiber $\Gr_{\calG,k}$ is the affine flag variety we've studied so far. Let $\mu$ be a conjugacy class of geometric coweights of $G$ with reflex field $E$, and consider the associated Schubert variety $\Gr_{G,E,\leq \mu}$. We define the local model $\Gr_{\calG,O_E,\leq \mu}$ as being the scheme-theoretic image of the former in $\Gr_{\calG,O_E}$.

\begin{theorem}\label{teo_seminormal_modelo_local}
	The semi-normal flat $O$-scheme $\mathrm{Gr}^{\mathrm{sn}}_{\calG,O_E,\leq\mu}$ has $\varphi$-split special fiber. In particular, it is normal, Cohen--Macaulay and $\varphi$-rational.
\end{theorem}

\begin{proof}
This result is proved in \cite[Theorems 3.8 and 3.9]{Zhu14} and \cite[Theorem 2.1]{HR22} for tame $G$, in \cite[Theorem 5.4]{FHLR22} when $p>2$ or $\mathrm{SU}_3$-free $G$.
We now prove this for all $G$ building on \Cref{teo_seminormal_schubert}. 
In \cite[Theorem 1.3]{GL22}, we proved that $\mathrm{Gr}_{\calG,O_E,\leq \mu}$ is unibranch, i.e., its normalization is a universal homeomorphism, by a nearby cycle calculation relying on the Wakimoto filtration of \cite[Theorem 4]{AB09}, compare also with \cite[Theorem 4.17]{ALWY23}. One sees that the inclusion in $\mathrm{Gr}_{\calG,O_E,\leq \mu}$ of the jet group orbit of any representative $\nu \in X_*(T)$ of $\mu$ is a universal homeomorphism onto its image by Zariski's main theorem and unibranchness of local models. Now, proper monomorphisms are closed immersions, but both schemes above are flat and reduced, so we conclude that the orbit map is an open immersion, compare with \cite[Corollary 2.14]{Ric16}. Joint with \cite[Theorem 6.12]{HR21}, compare also with \cite[Theorem 6.16]{AGLR22}, this implies that the special fiber of $\Gr_{\calG,O_E,\leq \mu}$ is generically reduced. By flatness, the special fiber of the normalization of the local model is also $S_1$, hence itself reduced by Serre's criterion. In particular, the special fiber of the normalization is covered by the $\varphi$-split, hence weakly normal, variety $\mathrm{Gr}_{\calG,k,\leq \mu}^{\mathrm{sn}}$ in light of \Cref{thm_splinter}, so this normalization equals $\mathrm{Gr}^{\mathrm{sn}}_{\calG,O_E,\leq \mu}$. On the other hand, covering $G$ by a $z$-extension and invoking \Cref{teo_seminormal_schubert}, we can show that the special fiber also maps to $\mathrm{Gr}_{\calG,k,\leq \mu}^{\mathrm{sn}}$. Since it is reduced, we deduce that this map is an isomorphism. 
As for the remaining assertions, we can derive them formally from the $\varphi$-splitness of the special fiber just like in \cite[Theorem 5.4]{FHLR22}.
\end{proof}

Note that during the above proof, we came really close to showing the embedded version of the previous normality theorem.

\begin{theorem}\label{teo_normal_mod_local_mergulhado}
	If $p\nmid \#\pi_1(G_{\mathrm{der}})$, then $\mathrm{Gr}_{\calG,O_E,\leq\mu}$ is normal.
\end{theorem}

\begin{proof}
	This follows from \Cref{teo_seminormal_schubert} and \Cref{teo_seminormal_modelo_local} as follows. The special fiber of $\mathrm{Gr}^{\mathrm{sn}}_{\calG,O_E,\leq \mu}$ is semi-normal, hence it embeds into the ind-scheme $\mathrm{Gr}^{\mathrm{sn}}_{\calG,k}$. Hence, the semi-normalization of the local model is an isomorphism in the generic fiber and a closed immersion in the special fiber, so it must be an isomorphism by flatness and Nakayama's lemma.
\end{proof}

Before concluding this section, we state a helpful version of \Cref{teo_seminormal_modelo_local} at the ind-scheme level.

\begin{corollary}\label{lem_ind_mod_loc_semi_norm}
Assume $G$ is simply connected. Then, $\Gr^{\mathrm{sn}}_{\calG,O}$ is an ind-scheme and its special fiber identifies with $\Gr_{\calG,k}^{\mathrm{sn}}$.	
\end{corollary}
\begin{proof}
	By étale descent, we may enlarge $k$ so that $G$ is quasi-split by Steinberg's theorem. Let $\rho^\vee$ be the half-sum of all coroots. It is clear that for simply connected $G$, we have $\mu \leq N\rho^\vee$ for $N\gg 0$. Also the definition field of $\rho^\vee$ is $F$. In particular, we see that $\Gr_{\calG,O}^{\mathrm{sn}}$ equals the colimit of the local models $\Gr^{\mathrm{sn}}_{\calG,O,\leq N\rho^\vee}$ in the category of $O$-sheaves. We must show that the transition morphisms of our presentation are closed immersions. This is true on geometric fibers by \Cref{thm_splinter,teo_seminormal_modelo_local}, so we conclude by flatness and Nakayama's lemma.
\end{proof}

\section{Distributions}

In this section, we study the notion of distributions of ind-schemes, which arise as non-linear differential operators near a given point. They capture essentially the same information as formal completions at a point, but with the added bonus that group distributions form a solid associative Hopf $k$-algebra. We give a Serre presentation of $\mathrm{Dist}(\Res_{F/k}G)$ in terms of its rank $1$ and unipotent parts, and use this to prove the normality theorem. 
\subsection{Preparations}
During this section, we work again in the category of pointed ind-schemes $(X,x)$ over a finite field $k$, i.e., $X$ is an arbitrary ind-scheme and $x$ is a $k$-valued point of $X$. Our notion of distributions is similar to \cite[Definition 7.1]{HLR18}, except we drop the Artinian condition and take topological information into account. 

\begin{definition}
The solid $k$-module $\mathrm{Dist}(X,x)$ is the filtered colimit of the solid $k$-modules $\mathrm{Hom}_k(\Gamma(Z,\calO_Z),k)$, as $(Z,x)$ runs over all closed nilpotent pointed $k$-subschemes of $(X,x)$.
\end{definition}

The $k$-module $\mathrm{Dist}(X,x)$ has a natural solid structure in the sense of Clausen--Scholze \cite[Proposition 7.5, Theorem 8.1]{CS19}, because $k$ is a finitely generated $\bbZ$-algebra, and we can write the global sections of $Z$ as the filtered colimit of its finitely generated $k$-submodules, so its $k$-module dual equals the cofiltered limit of its finitely generated $k$-quotients. 
Note that $\mathrm{Dist}(X,x)$ only depends on the formal completion of $(X,x)$: in fact, it equals the solid $k$-dual of the ring of formal sections. Later on, whenever $x$ is understood, we will just omit it from the notation.

Given a morphism $(X,x)\to (Y,y)$ of pointed ind-$k$-schemes, there is a natural map $\mathrm{Dist}(X,x)\to \mathrm{Dist}(Y,y)$. Indeed, any closed nilpotent subscheme $Z_X \subset X$ maps to $Y$ via a closed nilpotent subscheme $Z_Y \subset Y$, and this induces a morphism of solid $k$-modules $\mathrm{Dist}(Z_X,x)\to \mathrm{Dist}(Z_Y,y)$ that is natural in the various closed nilpotent subschemes. 
Note that closed immersions induce monomorphisms at the level of distributions.
In particular, we see that $\mathrm{Dist}(Z,z)$ embeds into $\mathrm{Dist}(X,x)$ for any closed nilpotent subscheme $Z \subset X$ supported at $x$.
In the next lemma, we impose no representability nor finiteness condition on the given morphism.

\begin{lemma}\label{lem_dist_iso_form_etale}
	Let $f\colon (X,x)\to (Y,y)$ be a map of pointed ind-schemes. Then, $\mathrm{Dist}(f)$ is an isomorphism if and only if $f$ is formally étale at $x$.
\end{lemma}
\begin{proof}
	First, we handle the if direction. By \Cref{lem_form_etale_completion}, formally étale maps induce an isomorphism between formal completions. The definition of $\mathrm{Dist}$ only depends on formal completions, so the result is clear.
	Now, we handle the only if direction. Note that the condensed $k$-dual of $\mathrm{Dist}(Z,z)$ equals its ring $\Gamma(Z,\calO_Z)$ of global sections. In particular, the ring of formal sections of $(X,x)$ equals the condensed $k$-dual of $\mathrm{Dist}(X,x)$. Since $\mathrm{Dist}(f)$ is an isomorphism, we conclude that $f$ induces an isomorphism of formal rings. Hence, our claim follows from \Cref{lem_form_etale_local_rings}. 
\end{proof}

We want to understand when maps of distributions are surjections in the category of solid $k$-modules. For this, we say that a map $f\colon (X,x)\to (Y,y)$ of pointed ind-schemes over $k$ is formally dominant if the scheme-theoretic image functor along $f$ preserves formal completions.
 
\begin{lemma}\label{lem_dist_surj_sch_dominant}
	Let $(X,x)\to (Y,y)$ be a formally dominant map of pointed ind-$k$-schemes. Then, $\mathrm{Dist}(f)$ is a surjective map of solid $k$-modules.
\end{lemma}
\begin{proof}
We may assume that $X$ is nilpotent and $Y$ is the scheme-theoretic image of $X$ along $f$ (and thus also nilpotent). At the level of formal sections, we have an inclusion $\Gamma(Y,\calO)\to \Gamma(X,\calO)$ of discrete solid $k$-modules. Upon taking $k$-duals, this turns into a surjection of solid $k$-modules. 
\end{proof}

If $f$ is a scheme-theoretic dominant map of finite type $k$-schemes, then it is also formally dominant by Chevalley's lemma, compare with \cite[Lemma 7.3]{HLR18}. However, beware that this is false as soon as we drop finiteness, as revealed by the endomorphism of the scheme $\Res_{O/k}\bbA^1_k$ given by $\sum t^iz_i\mapsto \sum t^iz_i^i$. Indeed, the scheme-theoretic images of nilpotent schemes along that endomorphism are always of finite type.

Next, we show that distribution modules are factorizable in products of ind-schemes.

\begin{lemma}\label{lem_dist_fact_flat}
	Let $(X,x)$ and $(Y,y)$ be pointed $k$-schemes. The canonical map of solid $k$-modules
	\begin{equation}
	\mathrm{Dist}(X,x) \otimes_k^\blacksquare \mathrm{Dist}(Y,y) \to \mathrm{Dist}(X\times Y, (x,y))
	\end{equation}
	is an isomorphism.
\end{lemma}

\begin{proof}
	Notice that if we let $Z_X$, resp.~$Z_Y$, run over all nilpotent thickenings of $x$ at $X$, resp.~$y$ at $Y$, then the colimit of the closed nilpotent subschemes $Z_X \times_k Z_Y$ recovers the formal completion of $X\times_kY$ at $(x,y)$. In particular, we may assume that $X$ and $Y$ are themselves nilpotent, as the solid tensor product preserves colimits. But since $k$ is a field, taking $k$-module duals commutes with tensor products, and it is easy to check that the isomorphism respects the solid structure.
\end{proof}

Note that $\mathrm{Dist}$ is a covariant functor, so the diagonal $\Delta\colon X\to X\times X$ yields a comultiplication map
\begin{equation}
	\mu:=\Delta_\ast\colon \mathrm{Dist}(X,x) \to \mathrm{Dist}(X\times X, (x,x)) \simeq \mathrm{Dist}(X,x) \otimes_k^\blacksquare \mathrm{Dist}(X,x)
\end{equation}
making the distribution $k$-module into a cocommutative solid $k$-coalgebra, with counit $\epsilon \colon k=\mathrm{Dist}(\ast) \to \mathrm{Dist}(X,x)$ induced by $x$. Assume $X=G$ is a $k$-group ind-scheme. 
We define its distribution $k$-algebra $\mathrm{Dist}(G):=\mathrm{Dist}(G,1)$ as the solid $k$-module of distributions based at the origin. The multiplication map $m\colon G\times G \to G$ induces a multiplication map
\begin{equation}
	m:=m_\ast \colon \mathrm{Dist}(G\times G)\simeq \mathrm{Dist}(G)\otimes_k^\blacksquare \mathrm{Dist}(G) \to \mathrm{Dist}(G)
\end{equation}
making the distribution $k$-module into a cocommutative associative solid Hopf $k$-algebra, with antipode induced by the inverse map of $G$. Beware that this associative algebra ought to be as commutative as $G$ itself, and hence it is very rarely so.
\subsection{A Serre presentation}

Let $G$ be a connected reductive $F$-group and assume it is residually split in the sense of \cite[Definition 9.10.11]{KP23} throughout this subsection. By \cite[Proposition 9.10.12]{KP23}, $G$ is also quasi-split. In particular, we can fix a pinning in the sense of \cite[Section 4.1]{BT84}, i.e., the data consisting of a maximally split maximal $F$-torus $T$ of $G$, a Borel $F$-subgroup $B\subset G$, and certain isomorphisms $x_\alpha^{-1}$ between $U_\alpha$ and certain explicit groups that we describe below. 

Our results in this section consist in giving a solid associative $k$-algebra presentation for $\mathrm{Dist}(\Res_{F/k}G)$ in terms of $\mathrm{Dist}(\Res_{F/k}U^\pm)$. We have a notion of coproduct in the category of solid associative $k$-algebras by taking the usual construction in condensed associative $k$-algebras and then solidifying it.

\begin{proposition}\label{prop_surj_dist_pm_simple_full}
	If $G$ is simply connected and residually split, the natural map \begin{equation}\label{eq_gen_pmsimple_full}\ast^\blacksquare_{\alpha \in \pm\Delta}\mathrm{Dist}(\Res_{F/k}U_{\alpha})\to \mathrm{Dist}(\Res_{F/k}G)\end{equation} of solid associative $k$-algebras is an epimorphism.
\end{proposition}	

\begin{proof}
	This will be proved along several computational lemmas that appear below. First, we note that we have a decomposition
	\begin{equation}\label{eq_fact_big_cell}
		\mathrm{Dist}(\Res_{F/k}U^-)\otimes_k^\blacksquare \mathrm{Dist}(\Res_{F/k}T) \otimes_k^\blacksquare \mathrm{Dist}(\Res_{F/k}U^+)=\mathrm{Dist}(\Res_{F/k}G)
	\end{equation}
by combining formal étaleness with \Cref{lem_dist_fact_flat}. In order to handle the unipotent parts of the distribution algebra, it is enough to show that 
\begin{equation}\label{eq_surj_unip_simple}
	\ast^\blacksquare_{\alpha \in \Delta}\mathrm{Dist}(\Res_{F/k}U_{\alpha})\to \mathrm{Dist}(\Res_{F/k}U^+)
\end{equation}
is a surjection of solid associative $k$-algebras, which will be done in \Cref{lem_gen_unip}. Finally, to handle the torus we observe that there is also a decomposition
\begin{equation}\label{eq_fact_torus}
	\mathrm{Dist}(\Res_{F/k}T)=\otimes_{\alpha \in \Delta}^\blacksquare\mathrm{Dist}(\Res_{F/k}T_\alpha)
\end{equation}
in the category of solid $k$-modules by \Cref{lem_dist_fact_flat}. Here, $T_\alpha$ denotes the intersection of $T$ with the subgroup $G_\alpha$ generated by $U_{\pm \alpha}$. In particular, this reduces our statement to the rank $1$ case, which is handled in \Cref{lem_surj_dist_unip_rank1}.
\end{proof}

Below, we perform the explicit computations required to verify the claims used in \Cref{prop_surj_dist_pm_simple_full}. For this, let us recall that are fixing the data of certain compatible isomorphisms as follows. Assume first that $2\alpha$ is not a root, and the pinning is of the form $x_\alpha \colon \Res_{F_\alpha/F}\bbG_{\mathrm{a},F_\alpha} \to U_\alpha$, where $F_\alpha/F$ is a separable field extension that is unique up to conjugation under $\mathrm{Gal}_F$. If $2\alpha$ is a root, then the pinning takes the form $x_\alpha \colon \Res_{F_{2\alpha}/F} \bbG_{\mathrm{p},F_\alpha/F_{2\alpha}}$, where $F_\alpha/F$ is a separable field extension, and $F_\alpha/F_{2\alpha}$ is a quadratic field extension. Here, the group $\bbG_{\mathrm{p},F_\alpha/F_{2\alpha}}$ is a three-dimensional $F_{2\alpha}$-group described explicitly in \cite[4.1.15]{BT84} whose rational points are pairs in $(F_\alpha)^2$ with trace-zero second coordinate.

\begin{lemma}\label{lem_gen_unip}
	The map \eqref{eq_surj_unip_simple} is a surjection of solid associative $k$-algebras.
\end{lemma}	
\begin{proof}
	Observe that $\mathrm{Dist}(U^+)$ is the product of the $\mathrm{Dist}(U_\alpha)$ as $\alpha$ runs over all positive roots by \Cref{lem_dist_fact_flat}. This statement is proved by induction on the height of a root, by exploiting the commutator relations inside $U^+$ explicitly written down in \cite[Addendum]{BT84}. It is then clear that we are reduced to proving the statement for almost simple simply connected groups of rank $2$ with $\Delta=\{\alpha,\beta\}$, an exhaustive list being given by restrictions of scalars of groups of type $A_2$, $C_2$, ${}^2A_3$, ${}^2A_4$, $^{3,6}D_4$, and $G_2$.
	
	First, if $G$ is of type $A_2$, then $\Phi^+$ contains exactly three roots $\alpha, \beta, \alpha +\beta$, then we get the commutator formula
	\begin{equation}
		[x_\alpha(z_1),x_\beta(z_2)]=x_{\alpha+\beta}(z_1z_2 ).
	\end{equation}
	It is clear that this map is formally dominant, because it has bounded coefficients and hits a topological $k$-basis of $F$. In particular, $\mathrm{Dist}(\Res_{F/k}U_{\alpha+\beta})$ is contained in the image of \eqref{eq_surj_unip_simple} by \Cref{lem_dist_surj_sch_dominant}.
	
	If instead $G$ is of type $C_2$ or ${}^2A_3$, then $\Phi^+$ contains exactly four roots of the form $\alpha, \beta, \alpha+\beta, 2\alpha+\beta$, and we get
	\begin{equation}
		[x_\alpha(z_{1}),x_\beta(z_2)]=x_{\alpha+\beta}(z_{1}z_2)x_{2\alpha+\beta}(N(z_1)z_2)
	\end{equation} where $N(z_1)=z_1^2$ or the norm of the quadratic extension in the non-split case, and the variables are understood to be formal. Again, one can check easily by this expression that the commutator map is formally dominant. In particular, we deduce by \Cref{lem_dist_surj_sch_dominant} that $\mathrm{Dist}(\Res_{F/k}U_{\alpha+\beta})$ and $\mathrm{Dist}(\Res_{F/k}U_{2\alpha+\beta})$ both lie in the image of \eqref{eq_surj_unip_simple}. 
	
	If $G$ is of type $^{3,6}D_4$ or $G_2$, then $\Phi^+=\{ \alpha, \beta, \alpha+\beta, 2\alpha+\beta, 3\alpha+\beta, 3\alpha+2\beta\}$, and we get
	\begin{equation}
		[x_\alpha(z_1),x_\beta(z_2)]=  x_{\alpha+\beta}(z_1z_2)x_{2\alpha+\beta}(\theta(z_1)z_2) x_{3\alpha+\beta}(N(z_1)z_2)x_{3\alpha+2\beta}(N(z_1)z_2^2),
	\end{equation}
	where $N(z_1)=z_1^3$ in the split case and is the usual norm of the fixed cubic extension in the non-split case, $\theta(z_1)z_1=N(z_1)$, and the variables are understood to be formal.
	In this case, the commutator is not dominant for dimension reasons. However, one can observe that the formal completion of $\Res_{F/k}U_{\alpha+\beta}$ is contained in the scheme-theoretic image of the above map, because $z_1z_2$ is algebraically idenpendent from the remaining polynomials. In particular, $\mathrm{Dist}(\Res_{F/k}U_{\alpha+\beta})$ is contained in the image of \eqref{eq_surj_unip_simple}. This means we can transport $x_\alpha(z_1z_2)$ to the left side of the equation, and we are reduced to the remaining three root groups. We continue this procedure first by showing containment of $\mathrm{Dist}(\Res_{F/k}U_{3\alpha+2\beta})$, then of $\mathrm{Dist}(\Res_{F/k}U_{2\alpha+\beta})$, and finally of $\mathrm{Dist}(\Res_{F/k}U_{3\alpha+\beta})$.
	
	Finally, we consider the case where $G$ is of type $^2A_4$. It follows that $\Phi^+=\{ \alpha, \beta, \alpha+\beta, 2\alpha, 2\alpha+\beta, 2\alpha+2\beta\}$. Therefore, we get
	\begin{equation}
		[x_{\alpha}(z_1,z_2),x_\beta(z_3)]=x_{\alpha+\beta}(\sigma(z_1z_3),N(z_1)z_2)x_{2\alpha+\beta}(z_3N(z_1)+z_3z_2)
	\end{equation}
	where $N$ denotes the norm of $F_\alpha/F_{2\alpha}$ and $\sigma$ the non-trivial involution. We can see that the formal completion of $\Res_{F/k}U_{\alpha+\beta}$ $\mathrm{Dist}(U_\Delta)$ is contained in the scheme-theoretic image of the above map. As above, this is enough to show that $\mathrm{Dist}(\Res_{F/k}U_{\alpha+\beta})$ and $\mathrm{Dist}(\Res_{F/k}U_{2\alpha+\beta})$ both lie in the image of \eqref{eq_surj_unip_simple}.
\end{proof}

In the next lemma, we handle the rank $1$ case. For this, we recall that there is a natural isomorphism $\alpha^\vee: \Res_{F_\alpha/F}\bbG_{\mathrm{m},F_\alpha} \to T_\alpha$, where the right side equals the intersection of $T$ with the subgroup $G_\alpha$ generated by $U_{\pm \alpha}$.

\begin{lemma}\label{lem_surj_dist_unip_rank1}
	Assume $G$ is simply connected, residually split and has rank $1$. Then, \eqref{eq_gen_pmsimple_full} is a surjection of solid associative $k$-algebras.
\end{lemma}	
\begin{proof}
	Let $\alpha$ be the positive simple root of $G$. First, we assume that $G=\Res_{F_\alpha/F}\mathrm{SL}_{2,F_\alpha}$. Inside $\Res_{F/k}G$, we have the following equation
	\begin{equation}\label{funct_eq_sl2}
		x_\alpha(z_1)x_{-\alpha}(z_2)=x_{-\alpha}\big(\frac{z_2}{1+z_1z_2}\big)\alpha^\vee(1+z_1z_2)x_\alpha\big(\frac{z_1}{1+z_1z_2}\big)
	\end{equation}
where the variables are understood to be formal.
If we isolate the term $\alpha^\vee(1+z_1z_2)$ on the right side, we see that the map is formally dominant. This implies by \Cref{lem_dist_surj_sch_dominant} that $\mathrm{Dist}(\Res_{F/k}T)$ is contained in the image of \eqref{eq_gen_pmsimple_full}. By the big cell factorization, this implies the desired surjectivity in the split case. 

	
	
	Next, we handle the rank $1$ quasi-split group $G=\Res_{F_{2\alpha}/F}\mathrm{SU}_{3,F_\alpha/F_{2\alpha}}$ where $F_\alpha/F$ is a separable field extension and $F_\alpha/F_{2\alpha}$ is quadratic. 
	In $\Res_{F/k}G$, we have the following equality
	\begin{equation}\label{funct_eq_su3}
		x_\alpha(z_1,z_2)x_{-\alpha}(z_3,z_4)=x_{-\alpha}(f_1,f_2)\alpha^\vee(1+g)x_{\alpha}(f_3,f_4)
	\end{equation}
where the $f_i$ are explicit rational functions on the $z_i$ involving the quadratic involution $\sigma$, which we omit for simplicity, and \begin{equation}
		g={-\sigma(z_1)z_3+(z_2+\lambda N(z_1))(z_4+\lambda N(z_3))},
	\end{equation} 
compare with \cite[4.1.12]{BT84}.
The map resulting from isolating $\alpha^\vee(g)$ in the right side can be checked to be formally dominant, so we see by \Cref{lem_dist_surj_sch_dominant} that $\mathrm{Dist}(\Res_{F/k}T)$ is contained in the image of \eqref{eq_gen_pmsimple_full}, implying that this is a surjection.
\end{proof}

Now, we have all the necessary tools at our disposal to establish a Serre presentation for the loop distribution algebra $\mathrm{Dist}(\Res_{F/k}G)$ of a simply connected $F$-group $G$, inspired by \cite[Proposition 3.6]{Tak83a} and \cite[Theorem 5.1]{Tak83b}. For this, we need to introduce the following notation: given $\alpha \in \Delta$, a positive simple root, we let $G_{\pm \alpha}$ be the derived subgroup generated by $U_{\pm \alpha}$, then we let $V_{\pm \alpha} \subset U^\pm$ be the unique smooth connected unipotent subgroup such that $U^\pm =U_{\pm \alpha}\times V_{\pm \alpha}$, and finally we denote by $Q_{\pm\alpha}$ the semi-direct product $G_{\pm\alpha}\ltimes U_{\pm \alpha}$. Let us emphasize that $Q_{\pm \alpha}$ is a kind of derived parabolic subgroup. Indeed, it equals the extension of the derived subgroup $G_{\pm \alpha}$ of the standard Levi of the minimal parahoric $P_{\pm \alpha}$ by its unipotent radical $U_{\pm \alpha}$.

\begin{theorem}\label{thm_serre_pres_loop}
	If $G$ is simply connected and residually split, the kernel of the surjection
	\begin{equation}\label{eq_unip_gen_full_gp}
		\mathrm{Dist}(\Res_{F/k}U^+)\ast^\blacksquare \mathrm{Dist}(\Res_{F/k}U^-)\to \mathrm{Dist}(\Res_{F/k}G)
	\end{equation}
of solid associative $k$-algebras is the solid ideal generated by the kernels of
\begin{equation}
	\mathrm{Dist}(\Res_{F/k}U^{\pm})\ast^\blacksquare \mathrm{Dist}(\Res_{F/k}U_{\mp\alpha})\to \mathrm{Dist}(\Res_{F/k}Q_{\pm\alpha}),
\end{equation}
as $\alpha\in \Delta$ runs through all positive simple roots.
\end{theorem}

\begin{proof}
	Let $\mathscr{U}$ be the solid associative $k$-algebra given by the generators and relations described in the statement of the theorem. We have a surjection $\mathscr{U} \to \mathrm{Dist}(\Res_{F/k}G)$ of solid associative $k$-algebras by \Cref{prop_surj_dist_pm_simple_full}. Our first goal is to define a section $s$ of the previous surjection in the category of solid $k$-modules. By definition, the algebra $\mathscr{U}$ contains $\mathrm{Dist}(U^\pm)$ and $\mathrm{Dist}(T_\alpha)$ for any $\alpha \in \Delta$ as solid associative $k$-subalgebras. This yields the desired section by using the factorizations \eqref{eq_fact_big_cell} and \eqref{eq_fact_torus}. 
	
	To finish the proof, it is enough by \Cref{lem_gen_unip} to prove stability of the corresponding solid $k$-submodule $\mathrm{im}(s)$ under left multiplication by every $\mathrm{Dist}(\Res_{F/k}U_{\alpha})$ for any root $\alpha \in \Delta$.
	Before doing this however, we prove that the solid module given as the product of the $\mathrm{Dist}(\Res_{F/k}T_\alpha)$ gives rise to a solid commutative $k$-subalgebra of $\mathscr{U}$. Let $\alpha\neq \beta \in \Delta$ be distinct positive simple roots. We claim that the natural conjugation map
	\begin{equation}
		\mathrm{Dist}(\Res_{F/k}T_\alpha) \otimes^\blacksquare \mathrm{Dist}(\Res_{F/k}G_\beta)\to \mathrm{Dist}(\Res_{F/k}G_\beta)
	\end{equation} lifts to $\mathscr{U}$, which clearly implies the desired commutativity. Indeed, $\mathrm{Dist}(\Res_{F/k}G_\beta)$ is surjected upon by the algebra coproduct of $\mathrm{Dist}(\Res_{F/k}U_{\pm \beta})$ according to \Cref{lem_surj_dist_unip_rank1}. The latter solid algebra carries a conjugation action by $\mathrm{Dist}(\Res_{F/k}T_\alpha)$ and this action is compatible with the maps to $\mathscr{U}$, because our universal solid algebra contains $\mathrm{Dist}(\Res_{F/k}Q_{\pm\alpha})$ as a solid subalgebra and $U_{\pm \beta},T_\alpha\subset Q_{\pm \alpha}$. 

Finally, we check stability of the section $s$ to the natural surjection $\mathscr{U}\to \mathrm{Dist}(\Res_{F/k}G)$ under left multiplication by $\mathrm{Dist}(\Res_{F/k}U_\alpha)$ for every $\alpha \in \Delta$. 
Writing $U^-=V_{-\alpha} \times U_{-\alpha}$, we deduce from \Cref{lem_dist_fact_flat} a decomposition \begin{equation}\mathrm{Dist}(\Res_{F/k}U^-)=\mathrm{Dist}(\Res_{F/k}V_{-\alpha})\otimes^\blacksquare \mathrm{Dist}(\Res_{F/k}U_{-\alpha}).\end{equation} On the other hand, we know that $\mathrm{Dist}(\Res_{F/k}U_\alpha)$ normalizes $\mathrm{Dist}(\Res_{F/k}V_{-\alpha})$ inside $\mathscr{U}$ due to our imposed relations coming from $\mathrm{Dist}(\Res_{F/k}Q_\alpha)$. In other words, we can switch the order in which $\mathrm{Dist}(\Res_{F/k}V_{-\alpha})$ and $\mathrm{Dist}(\Res_{F/k}U_\alpha)$ are multiplied inside $\mathscr{U}$.
Finally, we can assemble the product \begin{equation}\mathrm{Dist}(\Res_{F/k} U_\alpha)\otimes^\blacksquare \mathrm{Dist}(\Res_{F/k}U_{-\alpha})\subset \mathrm{Dist}(\Res_{F/k}
G_\alpha)\end{equation} inside $\mathscr{U}$, because $Q_\alpha$ contains $G_\alpha$, and then use the factorization \eqref{eq_fact_big_cell} for $G_{\alpha}$ to switch the order of the loop distributions of $\pm \alpha$. Since the solid commutative $k$-subalgebra $\mathrm{Dist}(\Res_{F/k}T)\subset \mathscr{U}$ normalizes $\mathrm{Dist}(\Res_{F/k}U_\alpha)$, we can finally pull this factor across and absorb it into $\mathrm{Dist}(\Res_{F/k}U^+)$, concluding our claim. 
\end{proof}

\subsection{Proof of normality}
In this section, we finally prove \Cref{teo_seminormal_schubert} (which also implies \Cref{teo_normal_mod_local_mergulhado}) using our newly acquired distribution skills. The first result needed is a surjectivity one, giving us some control on what happens with the distributions of the semi-normal loop group.

\begin{lemma}\label{lem_surj_dist_snloop}
If $G$ is residually split, the natural map of solid associative $k$-algebras
	\begin{equation}\label{eq_parah_jet_surj_dist}
		\ast^\blacksquare_{\alpha \in \Delta} \mathrm{Dist}(\Res_{O/k}\calG_{s_\alpha}) \to \mathrm{Dist}(\Res^{\mathrm{sn}}_{F/k}G)
	\end{equation}
is a surjection. Here, $\calG_{s_\alpha}$ denotes the parahoric $O$-model attached to a wall of a fixed alcove with associated Iwahori $O$-model $\calI$.
\end{lemma}
\begin{proof}
	Consider the Demazure resolution $\Gr_{\calI,\leq s_\bullet}\to \Gr^{\mathrm{sn}}_{\calI,\leq w}$ which is a proper birational cover. In particular, it surjects at the level of distributions supported at the identity by \Cref{lem_dist_surj_sch_dominant} and Chevalley's lemma, compare with \cite[Lemma 7.3]{HLR18}. This map lifts to the natural $\Res_{O/k}\calI$-torsors on the right as follows
	\begin{equation}
		\Res_{O/k}\calG_{s_1}\times^{\Res_{O/k}\calI}\times\dots\times^{\Res_{O/k}\calI}\Res_{O/k}\calG_{s_n}\to (\Res^{\mathrm{sn}}_{F/k}G)_{\leq w}
	\end{equation}
by multiplying the subgroups on the left inside $\Res_{F/k}^{\mathrm{sn}}G$.
Note that the left side is built out of parahoric jet groups and it surjects again by formal dominance after taking $\mathrm{Dist}$, see \Cref{lem_dist_surj_sch_dominant}. 
Passing to colimits, we get our claim.
\end{proof}

Our next step consists of a reduction to rank $1$ groups.
\begin{proposition}\label{prop_seminormal_reducao_rk1}
	Assume $G$ is simply connected and residually split. If \Cref{teo_seminormal_schubert} holds for every rank $1$ subgroup $G_\alpha$, then it also does for $G$.
\end{proposition}
\begin{proof}
It is enough by \Cref{lem_reduction_step_form_etale} to show that the semi-normalization map for $\Res_{F/k}G$ is formally étale, and this can be verified at the distribution level, i.e., if we show that
	\begin{equation}
		\mathrm{Dist}(\Res^{\mathrm{sn}}_{F/k}G) \to \mathrm{Dist}(\Res_{F/k}G)
	\end{equation}
	is an isomorphism, thanks to \Cref{lem_dist_iso_form_etale}.

	By \Cref{thm_serre_pres_loop}, we are reduced to showing that $\mathrm{Dist}(\Res_{F/k}^{\mathrm{sn}}G)$ satisfies the generators and relations described in that statement. First of all, note that the semi-normal ind-group scheme $\Res_{F/k}U^\pm$ naturally sits inside $\Res_{F/k}^{\mathrm{sn}}G$ as a closed subgroup by naturality of the semi-normalization functor. Next, we note that by \Cref{lem_dist_fact_flat} applied to the big cell of a parahoric $\calG$ fixing an alcove of the standard appartment, we have a factorization
	\begin{equation}
		\mathrm{Dist}(\Res_{O/k}\calG)=\mathrm{Dist}(\Res_{O/k}\calU^-)
\otimes^\blacksquare \mathrm{Dist}(\Res_{O/k}\calT)\otimes^\blacksquare \mathrm{Dist}(\Res_{O/k}\calU^+),	\end{equation}
where $\calU^\pm$ and $\calT$ denote the corresponding $O$-models of $U^\pm$ and $T$ sitting inside $\calG$ as smooth closed subgroups. Since the jet group $\Res_{O/k}\calT$ is also semi-normal by \Cref{lem_res_jet}, it naturally lifts to $\Res_{F/k}^{\mathrm{sn}}G$. Moreover, it decomposes as a product of the $\Res_{O/k}\calT_\alpha$ for $\alpha \in \Delta$, because $G$ is simply connected. By assumption, we also know that $\Res_{F/k}G_\alpha$ is semi-normal, so it lifts canonically to $\Res_{F/k}^{\mathrm{sn}}G$ compatibly with $\Res_{F/k}U_{\pm \alpha}$ and $\Res_{O/k}\calT_\alpha$. In particular, we deduce that
\begin{equation}\label{eq_unip_gen_full_sn_gp}
	\mathrm{Dist}(\Res_{F/k}U^+)\ast^\blacksquare\mathrm{Dist}(\Res_{F/k}U^-)\to \mathrm{Dist}(\Res_{F/k}^{\mathrm{sn}}G)
\end{equation}
is a surjection of solid associative $k$-algebras, by combining \Cref{lem_surj_dist_unip_rank1,lem_surj_dist_snloop}.
	
	Seeing as $\mathrm{Dist}(\Res_{F/k}^{\mathrm{sn}}G)\to \mathrm{Dist}(\Res_{F/k}G)$ is a surjection of solid associative $k$-algebras, it is enough to prove that the kernel of \eqref{eq_unip_gen_full_sn_gp} contains that of \eqref{eq_unip_gen_full_gp}. Due to \Cref{thm_serre_pres_loop}, this follows by observing that the semi-normal ind-scheme $\Res_{F/k}Q_{\pm\alpha}$ lifts canonically to $\Res_{F/k}^{\mathrm{sn}}G$ 
	by naturality of the semi-normalization functor and our assumed rank $1$ case.
\end{proof}

In the next corollary, we say that a simply connected group is $\mathrm{SU}_3$-free if all of its rank $1$ subgroups $G_\alpha$ over any unramified extension of $F$ are inner forms of restrictions of scalars of $\mathrm{SL}_2$.

\begin{corollary}\label{cor_su3_free_seminormal}
	\Cref{teo_seminormal_schubert} holds for all $\mathrm{SU}_3$-free groups.
\end{corollary}

\begin{proof}
	By \Cref{lem_reduction_step_form_etale}, we may assume that $G$ is simply connected. After enlarging $k$, we may also assume by étale descent that $G$ is residually split. If $G$ is $\mathrm{SU}_3$-free, then all its rank $1$ simply connected subgroups $G_\alpha$ are isomorphic to a restriction of scalars of $\mathrm{SL}_2$. By \Cref{prop_lci}, we know that the Schubert varieties of $\mathrm{SL}_2$ are normal, so we conclude by the previous \Cref{prop_seminormal_reducao_rk1}.
\end{proof}

The only problem now is that we didn't yet prove \Cref{teo_seminormal_schubert} for odd unitary groups, so now we must find a way of completing it, and this will be done via the rank $1$ case of \Cref{teo_seminormal_modelo_local}, for which we need an independent proof.

\begin{lemma}\label{lem_seminormal_modelo_local_rk1}
	\Cref{teo_seminormal_modelo_local} holds for $G$ of rank $1$ independently of \Cref{teo_seminormal_schubert}.
\end{lemma}
\begin{proof}
	We may and do assume that $G$ is residually split and adjoint, so as to get the full Iwahori--Weyl group and the full coweight lattice. 
	Let $\calI$ be a Iwahori $O$-model dilated from $\calG$. We note that $\Gr_{\calG,O_E,\leq \mu}^{\mathrm{sn}}$ arises as the Stein factorization of the map
	\begin{equation}\label{eq_map_conv_parahoric}
		\Gr_{\calI,O_E,\leq \mu}^{\mathrm{sn}}\to \Gr_{\calG,O_E}
	\end{equation} because its fibers are geometrically connected and induce separable field extensions. Assume that the \Cref{teo_seminormal_modelo_local} holds for $\calI$ and $\mu$. Then, the geometric fibers of \eqref{eq_map_conv_parahoric} have vanishing higher direct images, as seen by the rationality of semi-normal Schubert varieties, see \Cref{thm_splinter}, and induction to treat unions of those like in \cite[Lemma 4.22]{FHLR22}. We deduce that the Stein factorization commutes with base change by \cite[Proposition 3.13]{Gor03}. As $\varphi$-splitness is preserved under proper pushforward, it is enough to treat the case of a Iwahori $O$-model $\calI$.
	
	Similarly, given a sequence $\mu_\bullet$ of conjugacy classes of geometric coweights of $G$ with reflex field $E_\bullet$, we consider the convolution map
	\begin{equation}\label{eq_map_conv_sequence}
		\Gr_{\calI,O_{E_\bullet},\leq \mu_\bullet}^{\mathrm{sn}}\to \Gr_{\calI,O_{E_\bullet}}, 
	\end{equation} 
	whose Stein factorization equals $\Gr^{\mathrm{sn}}_{\calI,O_{E_\bullet},\leq \mu}$ where $\mu$ denotes the sum of the $\mu_i$. Assume that \Cref{teo_seminormal_modelo_local} holds for $\calI$ and $\mu_i$. Then, we can deduce again from the rationality of Schubert varieties, see \Cref{thm_splinter}, the vanishing of the higher direct images of the geometric fibers of \eqref{eq_map_conv_sequence}, and hence the Stein factorization commutes with base change, again by \cite[Proposition 3.13]{Gor03}. Using this and flat descent along $O_E \to O_{E_\bullet}$, we conclude that $\Gr_{\calI,O_{E},\leq \mu}^{\mathrm{sn}}$ has $\varphi$-split special fiber (here, we applied reducedness of the special fiber, see \cite[Theorem 1.3]{GL22}, to know that the semi-normal local model is stable under base change of DVRs).
	In particular, we have reduced the problem to tiny geometric conjugacy classes $\mu$ of coweights in the sense of \cite[Subsection 7.1]{AGLR22}.
	
	Now, finally we perform some calculations for a Iwahori model $\calI$ and tiny $\mu$. We are allowed to replace our adjoint group $G$ by a $z$-extension with simply connected derived subgroup and $\mu$ by an arbitrary lift to this new group. By \Cref{thm_splinter}, we would know that \Cref{teo_seminormal_modelo_local} holds in rank $1$ as soon as $\Gr_{\calI,k,\leq \mu}$ is semi-normal. Since its dimension is at most $2$, it suffices to show smoothness of the Schubert varieties $\Gr_{\calI,\leq w}$ where $\ell(w)\leq 2$, and that their intersections are reduced. Translating to the neutral component, we may assume that $G$ is simply connected and $w$ can assume the values $1$, $s_0$, $s_1$, $s_0s_1$ or $s_1s_0$ where the $s_i$ are the distinct simple reflections. The Demazure variety $\Gr_{\calI,\leq (s_0,s_1)}$ is smooth and its tangent space maps bijectively to the $2$-dimensional $k$-vector space $(\mathrm{Lie}(\calG_{s_0})+\mathrm{Lie}(\calG_{s_1}))/\mathrm{Lie}(\calI)$. This shows formal étaleness at the identity and implies our claim. 
	
	Next, to treat reducedness of the intersections, we focus on the smooth $k$-surfaces $\Gr_{\calI,\leq s_0s_1}$ with $\Gr_{\calI,\leq s_1s_0}$. We leave the case $G=\Res_{F_\alpha/F}\mathrm{SL}_{F_\alpha/F}$ to the reader, because it is easy to verify and not really needed anyway. So, we have $G=\Res_{F_{2\alpha}/F}\mathrm{SU}_{3,F_\alpha/F_{2\alpha}}$ and can write local sections of $\Res_{O/k}\calG_{s_i} \to \Gr_{\calI,\leq s_i}$ as follows: either $x_\alpha(z_1,0)$ with $z_1 \in \nu k \subset F_\alpha^0$ if $i=0$, or $x_{-\alpha}(0,z_2)$ with $z_2 \in \mu k \subset F_\alpha$ when $i=1$, under the condition that the inverse of $\lambda N(\nu)\mu$ is a prime element of $O_{\alpha}$, see \cite[4.3.5]{BT84}. If we regard $z_1$ and $z_2$ as formal variables, we can look at equation \eqref{funct_eq_su3} and we deduce that the product $x_\alpha(z_1,0)x_{-\alpha}(0,z_2)$ defines an element of $\Gr_{\calI,\leq s_1s_0}$ only if $g$ is an integer and $\lambda N(f_1)$ is divisible by $\nu$. Note that $g=\lambda N(z_1)z_2$ and $f_1=-z_1z_2$ under our assumptions. The first condition implies that $z_1^2z_2$ vanishes, whereas the second implies that $z_1z_2$ vanishes. This equation spreads out to a reduced locally closed subscheme of $\bbA^2_k$, and it implies that the intersection of $\Gr_{\calI,\leq s_0s_1}$ with $\Gr_{\calI,\leq s_1s_0}$ equals the union of $\Gr_{\calI,\leq s_0}$ and $\Gr_{\calI,\leq s_1}$. 
\end{proof}

By \Cref{lem_ind_mod_loc_semi_norm}, we know that the semi-normal $O$-loop group $\Res_{O^2_\circ/O}^{\mathrm{sn}}\calG$ is an ind-scheme and its special fiber is semi-normal, i.e., it identifies with our original loop group $\Res^{\mathrm{sn}}_{F/k}G$ over $k$. Once again, we must show that the natural morphism
\begin{equation}\label{eq_seminormalization_bdloop}
	\Res_{O^2_\circ/O}^\mathrm{sn}\calG \to \Res_{O^2_\circ/O}\calG
\end{equation}
is a formally étale map of formally $O$-flat ind-schemes, by the $O$-flat version of \Cref{lem_form_etale_completion}. This has to be true in light of \Cref{teo_normal_mod_local_mergulhado}, but if we prove it independently for some parahoric model $\calG$, then it would also yield \Cref{teo_seminormal_schubert} for its generic fiber $G$.

In order to do this, we define again $\mathrm{Dist}(X,x)$ for a pointed ind-scheme over $O$ as the solid $O$-dual of the ring of formal sections, where we regard $O$ with its $t$-adic analytic structure. Note that $\mathrm{Dist}(X,x)$ only depends on the flat closure of the formal completion. We can see that most of our general lemmas on distributions continue to be true over $O$ provided that we restrict to maps of formally flat pointed ind-schemes over $O$, most notably \Cref{lem_dist_iso_form_etale,lem_dist_fact_flat}. However, there is no analogue of \Cref{lem_dist_surj_sch_dominant}, because injections of solid $O$-modules do not necessarily dualize to surjections, e.g., if the cokernel has torsion. 
In our concrete situation, we can remedy this failure thanks to the following explicit calculation with root groups.

\begin{proposition}\label{prop_gen_unip_bd_sc_red_quot}
	Assume $G$ is residually split and that the reductive quotient of $\calG_k$ is simply connected. Then, the natural map
	\begin{equation}
		\mathrm{Dist}(\Res_{O^2_\circ/O}\calU^+)\ast^\blacksquare_O \mathrm{Dist}(\Res_{O^2_\circ/O}\calU^-)\to \mathrm{Dist}(\Res_{O^2_\circ/O}\calG)
	\end{equation}
is a surjection of solid associative $O$-algebras.
\end{proposition}

\begin{proof}
Note that the big cell induces a decomposition
	\begin{equation}
		\mathrm{Dist}(\Res_{O^2_\circ/O}\calG)=\mathrm{Dist}(\Res_{O^2_\circ/O}\calU^-)\otimes_O^\blacksquare \mathrm{Dist}(\Res_{O^2_\circ/O}\calT)\otimes_O^\blacksquare \mathrm{Dist}(\Res_{O^2_\circ/O}\calU^+)
	\end{equation}
of solid $O$-modules by the $O$-flat version of \Cref{lem_dist_fact_flat} and formal flatness of each of the terms, compare with \Cref{lem_flat_bdloop}. Similarly, we also have a factorization
\begin{equation}
	\mathrm{Dist}(\Res_{O^2_\circ/O}\calT)=\otimes_{O,\alpha\in \Delta}^\blacksquare \mathrm{Dist}(\Res_{O^2_\circ/O}\calT_\alpha)
\end{equation}
as solid $O$-modules, owing to the simply connectedness of $G$, and a combination again of \Cref{lem_dist_fact_flat,lem_flat_bdloop}. In other words, it is clear that we can reduce to rank $1$ groups $G$. After enlarging $k$, these are either isomorphic to a restriction of scalars of $\mathrm{SL}_2$ or $\mathrm{SU}_3$. In the first case, $\calG=\Res_{O_\alpha/O}\mathrm{SL}_2$ and it is easy to check (and not really necessary for our proof of \Cref{teo_seminormal_schubert} anyway) the surjectivity, so we leave this task to the reader.

If $G=\Res_{F_{2\alpha}/F}\mathrm{SU}_{3,F_\alpha/F_{2\alpha}}$, then we can describe the special parahoric model $\calG$ with simply connected reductive quotient as follows. Let $\lambda \in F_\alpha^{1}$ be a trace $1$ element of maximal valuation, $\mu \in F_\alpha^0$ be an arbitrary element of trace $0$. Then, by \cite[4.3.5]{BT84} the closed subgroups $\calU_{\pm \alpha} \subset \calG$ are the unique smooth connected $O$-models of $U_{\pm \alpha}$ whose integral points are given by the subset $\nu_{\pm}O_\alpha \times \mu^{\pm 1}O_{2\alpha}\subset F_\alpha \times F_{\alpha}^0$, where $\nu_\pm$ is any element such that $\lambda N(\nu_\pm)$ equals $\mu^{\pm 1}t_\alpha$ up to a unit, where $t_\alpha$ is a uniformizer for $O_\alpha$.

Write the base change of the punctured disk $O^2_\circ$ along $O\to O_\alpha$ (resp.~$O_{2\alpha}$) as $R_\alpha$ (resp.~$R_{2\alpha}$). Note that we have an exchange equation \eqref{funct_eq_su3} at the formal completions, which expresses the middle term $\alpha^\vee(1+g)$ as a product of points of $\calU_{\pm \alpha}$. If we set $z_1=z_3=0$, then $g$ simplifies to $z_2z_4$. Here, $z_2\in \mu R_{2\alpha},z_4 \in\mu^{-1}R_{2\alpha}$, so we have obtained the distributions with coefficients in $R_{2\alpha}$. Similarly, if we set $z_1=z_4=0$, then $g$ simplifies to $\lambda z_2N(z_3)$ with $z_2 \in \mu R_{2\alpha}$ and $z_3\in \nu_-R_\alpha$, so we get the distributions with coefficients in $t_\alpha R_{2\alpha}$. Clearly, $R_\alpha$ is a free $R_{2\alpha}$-module with basis $\{1,t_\alpha\}$, so we deduce the desired surjectivity.
\end{proof}

\begin{corollary}\label{cor_seminormal_all}
	\Cref{teo_seminormal_schubert} holds for all groups.
\end{corollary}

\begin{proof}
	We may and do assume that $G$ is simply connected by \Cref{lem_reduction_step_form_etale}. By étale descent, we may also assume that $G$ is residually split. Furthermore, by \Cref{prop_seminormal_reducao_rk1} and \Cref{cor_su3_free_seminormal}, it suffices to treat the case where $G$ is a restriction of scalars of $\mathrm{SU}_3$. Let $\calG$ be a special parahoric model of $G$ such that $\calG_k$ has simply connected reductive quotient, whose existence is ensured by \cite[Lemma 4.11]{FHLR22}. Notice that \eqref{eq_seminormalization_bdloop} is a map of formally flat ind-schemes by \Cref{lem_ind_mod_loc_semi_norm} and \Cref{lem_flat_bdloop}. Indeed, for the semi-normal loop group, it is already flat from the beginning and this property passes to the formal completion using that excellent normal local rings are analytically irreducible. 
	
	We are going to show that the semi-normalization map over $O$ is formally étale. It is an isomorphism over $F$ by étale descent applied to \Cref{cor_su3_free_seminormal} for the split form $\mathrm{SL}_3$ of $G$. In particular, the map
	\begin{equation}
		\mathrm{Dist}(\Res_{O^2_\circ/O}^{\mathrm{sn}}\calG)\to \mathrm{Dist}(\Res_{O^2_\circ/O}\calG)
	\end{equation}
of torsion-free solid $O$-modules is a monomorphism. But \Cref{prop_gen_unip_bd_sc_red_quot} tells us that it is also surjective, as unipotent loop groups lift to the semi-normalization. We deduce formal étaleness by the $O$-flat variant of \Cref{lem_dist_iso_form_etale}. 
Passing to the special fibers, we deduce that $\Res_{F/k}^{\mathrm{sn}}G\to \Res_{F/k}G$ is formally étale, thanks to \Cref{lem_seminormal_modelo_local_rk1}. 
\end{proof}

	\bibliography{biblio.bib}
	\bibliographystyle{alpha}	
\end{document}